\documentclass[a4paper, reqno, english]{smfart}

\usepackage{amsfonts, amsthm, amsmath, amscd, amssymb}

\numberwithin{equation}{section}

\renewcommand\AA{\mathbb{A}}
\newcommand\QQ{\mathbb{Q}}

\newcommand\NN{\mathbb{N}}
\newcommand\RR{\mathbb{R}}

\newcommand\R{\mathcal{R}}
\newcommand\mcal{\mathcal}

\newcommand\UU{\mathcal{U}}
\renewcommand{\SS}{\mathcal{S}} 

\newcommand\ZZ{\mathbb{Z}}

\newcommand\PP{\mathbb{P}}
\newcommand\x{\mathbf{x}}
\newcommand\y{\mathbf{y}}
\newcommand\dd{\mathbf{d}}
\newcommand\kk{\mathbf{k}}
\renewcommand\d{\,\mathrm{d}}

\DeclareMathOperator{\hcf}{gcd}

\DeclareMathOperator{\Spec}{Spec}
\DeclareMathOperator{\disc}{disc}
\renewcommand{\le}{\leqslant}
\renewcommand{\ge}{\geqslant}
\renewcommand{\leq}{\leqslant}
\renewcommand{\geq}{\geqslant}

\newcommand\e{\eta}
\newcommand\ve{\varepsilon}
\newcommand\D{\Delta}
\newcommand\DD{\mathfrak{D}}

\renewcommand\phi{\varphi}

\newcommand\phid{\phi^\dagger}

\newcommand\al{\alpha}
\newcommand\be{\beta}

\renewcommand{\rho}{\varrho}
\newcommand{\bd}{\mathbf{d}}

\DeclareMathOperator{\Res}{Res}
\newcommand{\ma}{\mathbf}
\newcommand{\M}{\mathbf{M}}
\DeclareMathOperator{\Mod}{mod}
\renewcommand{\bmod}[1]{\,(\Mod{ #1})}

\newcommand{\cU}{\mathcal{U}}

\newcommand{\cA}{\mathcal{A}}
\renewcommand{\v}{\mathbf{v}}

\newcommand\sfl{\mathsf{\Lambda}}

\newtheorem{thm}{Theorem}
\newtheorem*{thm*}{Theorem}
\newtheorem{lemma}{Lemma}

\theoremstyle{definition}
\newtheorem*{ack}{Acknowledgements}

\DeclareMathOperator{\vol}{vol}

\begin{document}

\title[Binary forms and Ch\^atelet surfaces]{Binary forms as sums of two squares and Ch\^atelet surfaces}

\author{R. de la Bret\`eche}
\address{
Institut de Math\'ematiques de Jussieu\\
Universit\'e Paris Diderot-Paris 7\\
UFR de Math\'ematiques\\
Case 7012\\
B\^atiment Chevaleret\\
75205 Paris cedex 13\\ 
France}
\email{breteche@math.jussieu.fr}

\author{T. D. Browning}
\address{School of Mathematics\\
University of Bristol\\ Bristol\\ BS8 1TW\\ United Kingdom}
\email{t.d.browning@bristol.ac.uk}

\date{\today}


\begin{abstract}
The representation of integral binary forms as sums of two squares is
discussed and applied to establish the
Manin conjecture for
certain Ch\^atelet surfaces over $\QQ$.
\end{abstract}

\maketitle
\tableofcontents

\section{Introduction}\label{s:intro}

Let $X$ be a proper smooth model of the affine surface 
\begin{equation}\label{eq:chat}
y^2-az^2=f(x), 
\end{equation}
where $a\in \ZZ$ is not a square and $f\in \ZZ[x]$ is a
polynomial of degree $3$ or $4$ without repeated roots. 
This defines a Ch\^{a}telet surface over $\QQ$
and we will be interested here
in providing a quantitative description of the density of
$\QQ$-rational points on $X$.
The anticanonical linear system $|-K_X|$ 
has no base point and gives a morphism
$\psi:X \rightarrow \PP^4$. 
This paper is motivated by a conjecture of Manin 
\cite{f-m-t} applied to the counting function
$$
N(B)=\#\{x \in X(\QQ): (H_4\circ \psi)(x)\leq B\},
$$
for a suitably metrized exponential height 
$H_4:\PP^4(\QQ)\rightarrow \RR_{>0}$, whose precise  
definition we will delay until \S \ref{s:UT}.
The conjecture predicts that $N(B)\sim c_X B(\log B)^{r_X-1}$
for some constant $c_X>0$, where $r_X$
is the rank of the Picard group associated to $X$. 
Peyre \cite{p} 
has given a conjectural interpretation of the constant $c_X$.

Getting an upper bound for $N(B)$ is considerably easier and the
second author \cite{tb} has shown that 
$N(B)\ll B(\log B)^{r_X-1}$ for any Ch\^atelet surface. 
When suitable assumptions are made on $a$ and $f$ in \eqref{eq:chat}
one can go somewhat further. Henceforth we assume that $a=-1$.
In recent joint work of the authors with Peyre
\cite{isk}, the Manin conjecture is confirmed for a family of  Ch\^atelet
surfaces that corresponds to $f(x)$ splitting
completely into linear factors over $\QQ$ in \eqref{eq:chat}. 
Our aim in the present investigation is to better understand the 
behaviour of $N(B)$ when the factorisation of $f(x)$ into irreducibles
contains an irreducible  polynomial of degree $3$. Here, as throughout
this paper, we take irreducibility to mean irreducibility over $\QQ$.
In this case it follows from the work of Colliot-Th\'el\`ene, Sansuc
and Swinnerton-Dyer \cite{ct1,ct2} that 
$X$ is $\QQ$-rational and so satisfies weak approximation.  
Moreover it is straightforward to calculate that 
$r_X=2$ (see \cite[Lemma 1]{tb}, for example). With this in mind we
see that the following result confirms the Manin prediction.

\begin{thm}\label{main}
We have $
N(B)\sim c_X B\log B,
$ 
as $B\rightarrow \infty$, 
where $c_X$ is the constant predicted by Peyre. 
\end{thm}

Our result bears comparison with recent work of 
Iwaniec and Munshi \cite{munshi}, where a counting function analogous
to $N(B)$ is studied as $B\rightarrow \infty$. However, using methods
based on the Selberg sieve, they are only able to produce a lower
bound for the counting function which is essentially of the correct order of
magnitude, a deficit that is remedied by our result.

Fix a constant $c>0$ once and for all. 
 We will work with compact subsets
$\mcal{R} \subset \RR^2$ 
whose boundary is a
piecewise continuously
differentiable closed curve of length
$$  \partial(\mcal{R})\leq c
\sup_{\x=(x_1,x_2)\in\mcal{R}}\max\{|x_1|,|x_2|\}= c r_\infty,
$$
say. For any
parameter $X>0$ let $X\mcal{R}= \{X\x: \x\in \mcal{R}\}$.
Our proof of the theorem relies upon estimating the sum 
$$
S(X)=\sum_{\x\in\ZZ^2\cap X\mcal{R}}r\big(L(\x)\big) r\big(C(\x)\big),
$$
where $r$ 
denotes the sum of two squares function, and $L,C$ are suitable binary forms
of degree $1$ and $3$, respectively, that are defined over $\ZZ$. 
Recall that $r(n)=4\sum_{d\mid n}\chi(d)$, where 
$\chi$ is the non-principal character modulo $4$.
For any $\bd=(d_1,d_2)\in \NN^2$ we let 
  \begin{equation}\label{defrho}
\rho(\bd)=\rho(\bd;L,C)=\# 
\{ \x \in \ZZ^2\cap [0,d_1d_2)^2 : d_1 \mid L (\x), ~ d_2\mid C(\x)\}.
\end{equation}
Furthermore, we
define $\mcal{E}$ to be the set of $m\in \NN$ such that there exists 
$\ell\in\ZZ_{\geq 0}$ for which 
$m\equiv 2^{\ell} \bmod{2^{\ell+2}}$. We denote by
$\mcal{E} \bmod{2^n}$ the projection of $\mcal{E}$ modulo $2^n$.
The following result forms the technical core of this paper.

\begin{thm}\label{mainS}
Let $\ve>0$ and let $\eta=1-\frac{1+\log\log 2}{\log 2}>0.086$.
Let $C\in \ZZ[\x]$ be an irreducible cubic form and let $L\in
\ZZ[\x]$ be a non-zero linear form. 
Assume that $L(\x)>0$ and $C(\x)>0$ for every $\x
\in \mcal{R}$. Then we have 
$$
S(X)
= \pi^2 \vol(\mcal{R})X^2\prod_p K_p   +
O\big(X^2(\log X)^{- \eta+\varepsilon}\big),
$$ 
where 
$$
K_p=
\Big(1-\frac{\chi(p)}{p}\Big)^2
\sum_{\nu_1,\nu_2\geq 0} 
 \frac{\chi(p^{\nu_1+\nu_2}) \rho(p^{\nu_1},p^{\nu_2})
 }{p^{2\nu_1+2\nu_2 }}
$$
if $p>2$ and 
$$
K_2=4\lim_{n\to \infty}2^{-2n}   
\#\left\{\x \in (\ZZ/2^n
\ZZ)^2: 
\begin{array}{l}  L(\x)\in  \mcal{E}\bmod{2^n}  \\   C(\x)\in
 \mcal{E}\bmod{2^n}\end{array}\right\}.
$$
The implied constant in this estimate is allowed to depend on $\ve,
L,C$ and $r_\infty$.
\end{thm}

The sum $S(X)$ is directly linked to the density of integral
points on the affine variety
$$
L (\x)=s_1^2+t_1^2,\quad C(\x)=s_2^2+t_2^2.
$$
Arguing along similar lines to the proof of
\cite[Theorem 4]{4linear}, one can interpret the leading constant 
in our estimate
for $S(X)$ as a product of local densities 
for this variety.  
In fact this variety is related to a certain intermediate
torsor that parametrises rational points on the Ch\^atelet surfaces
under consideration in this paper.

The asymptotic formula in Theorem \ref{mainS} should be taken as part
of an ongoing programme to   
understand the average order of arithmetic functions 
running over the values of binary quartic forms. 
One of the starting points for this topic lies in the work of 
Daniel \cite{D99}, where the analogue of $S(X)$ is estimated
asymptotically with $r(L)r(C)$ replaced by 
$r(x_1^4+x_2^4)$. A treatment of $r(L_1)\cdots r(L_4)$ 
for non-proportional linear forms $L_1,\ldots, L_4$ has been
accomplished by Heath-Brown
 \cite{h-b03}, which in turn has been improved by the authors
 \cite{4linear}.   Moreover, our allied
investigation \cite{L1L2Q} could easily be adapted to handled the
analogue of $S(X)$ featuring $r(L_1)r(L_2)r(Q)$
when $L_1,L_2$ are non-proportional linear forms and $Q$ is an irreducible
binary quadratic form. Dealing with
$r(Q_1)r(Q_2)$, for non-proportional irreducible quadratic forms
$Q_1,Q_2$, or 
even $r(F)$ for a general irreducible quartic form $F$, seems to present a
more serious challenge.

\begin{ack} 
It is pleasure to thank the referee for carefully reading the
manuscript and making numerous helpful comments, including drawing our
attention to an 
oversight in the original treatment of Lemma~\ref{lem:UU-upper}.
While working on this paper the second author was supported by EPSRC
grant number \texttt{EP/E053262/1}. Part of this work was carried out
while the second author was visiting the first author at 
the {\em Universit\'e Paris 7 Denis Diderot}, funded by 
ANR ``Points entiers points rationnels''.
\end{ack}

\section{Polynomials modulo $n$}

Our analysis will require information about the number of solutions to
various systems of polynomial equations modulo $n$. For any 
polynomial $f\in \ZZ[x]$ of degree $d\geq 2$, we define  
the content of $f$ to  be the greatest common divisor of its  
coefficients. Thus a polynomial has content $1$ if and only  
if it is primitive.  Let  
\begin{equation}
  \label{eq:f}
\rho_f(n)=\#\{ x\in\ZZ/n\ZZ:  f(x)\equiv 0\bmod n\}.
\end{equation}
Since $\rho_f(n)$ is a multiplicative function of $n$ it will suffice
to analyse it for prime powers.
We begin by recording the following upper bounds.

\begin{lemma}\label{dan2}  
Assume that $\disc(f)\neq 0$ and that $p$ is a prime which does not
divide the content of $f$, with $p^\mu \| \disc(f)$.  Then for any
$\nu \geq 1$ we have 
$$
\rho_f(p^{\nu}) \leq 
d\min\big\{p^{\frac{\mu}{2}}, p^{(1-\frac{1}{d})\nu}, p^{\nu-1}\big\}.
$$
\end{lemma}

\begin{proof} 
The inequality 
$\rho_f(p^{\nu}) \leq  d  p^{\frac{\mu}{2}}$ is 
due to Huxley \cite{hux} and  the inequality
$\rho_f(p^{\nu}) \leq  d  p^{(1-\frac{1}{d})\nu}$ is 
due to Stewart \cite[Corollary 2]{S91}. The final inequality 
is trivial. 
\end{proof}

One of the ingredients in our work will be the Dedekind zeta function
$$
\zeta_{k}(s)=
\sum_{\mathfrak{a}}\frac{1}{N_{k/\QQ}(\mathfrak{a})^s}=
\prod_{\mathfrak{p}}\Big(1-\frac{1}{N_{k/\QQ}(\mathfrak{p})^s}\Big)^{-1},
$$
for $\Re e (s)>1$, when $k$ is a number field obtained by adjoining to
$\QQ$ the root of an irreducible polynomial $f\in \ZZ[x]$.
Here $\mathfrak{a}$ runs over the set of integral ideals in $k$ and
$\mathfrak{p}$ runs over prime ideals.
By a well-known principle due to Dedekind \cite[p. 212]{ded}, for a
rational prime $p\nmid f_0\disc(f)$,  
where $f_0$ denotes the  
leading coefficient of $f$,  
we have the ideal factorisation
$
(p)=\mathfrak{p}_1^{e_1}\mathfrak{p}_2^{e_2}\cdots,
$
with $N_{k/\QQ}(\mathfrak{p}_i)=p^{r_i}$,  corresponding to the factorisation
$$
f(x)\equiv f_1(x)^{e_1}f_2(x)^{e_2}\ldots   \bmod p
$$
for polynomials $f_i(x)$ of degree $r_i$ which are irreducible modulo
$p$.  When $r_i=1$ the polynomial $f_i$ has a root modulo
$p$. Thus, for $p\nmid f_0 \disc(f)$,  
we have  
$$
\rho_f(p)=\#\{\mathfrak{p}: N_{k/\QQ}(\mathfrak{p})=p\}.
$$
The Eulerian factors of $\zeta_k(s)$ which correspond to prime ideals 
$\mathfrak{p}$ for which $N_{k/\QQ}(\mathfrak{p})=p^r$ for $r\geq 2$,
or 
$p\mid f_0 \disc(f)$, define a holomorphic and bounded function 
in the half-plane $\Re e (s)>\frac{1}{2}$, without any zeros there.

We will need to investigate the Dirichlet series 
\begin{equation}\label{defHf}
G_f(s) =\sum_{n=1}^\infty \frac{\rho_f(n)}{n^s}, \quad
G_f(s,\chi) =\sum_{n=1}^\infty \frac{\chi(n)\rho_f(n)}{n^s},
\end{equation}
for $\Re e(s)>1$, where $\chi$ is the real non-principal character
modulo $4$.  
Let $\kappa\in (0,\frac{1}{d})$.  
It follows from Lemma \ref{dan2} that for any  
$p\mid f_0\disc(f)$ we have 
$$
\sum_{\nu\geq 1}\frac{\rho_f(p^\nu)}{p^{\nu(1-\kappa)}}\ll_{\kappa}1.
$$
Hence for all $\kappa\in (0,\frac{1}{d})$ 
there exists an arithmetic function $h$ such that
$$
G_f(s)=\zeta_{k}(s)
\sum_{n=1}^\infty\frac{h(n)}{n^s}=\zeta_{k}(s)H_f(s),
$$
say, with $\sum_{n=1}^\infty |h(n)|n^{-1+\kappa}\ll_\kappa 1$.
In the same manner $G_f(s,\chi)$ is related to the Hecke $L$-function
$$
L(s,\chi)=\sum_{\mathfrak{a}}\frac{\chi(N_{k/\QQ}(\mathfrak{a}))}{N_{k/\QQ}(\mathfrak{a})^s},
$$
defined for $\Re e (s)>1$.  
Note that when $d$ is odd $L(s,\chi) $ will be analytic at $s=1$ since  
$\chi$ is a quadratic character.   
Thus we have $G_f(s,\chi)=L(s,\chi)H_f(s,\chi)$, where
$$
H_f(s,\chi)=\sum_{n=1}^\infty\frac{\chi(n)h(n)}{n^s}.
$$
The following result is well-known and follows on combining the above
with the results contained in the survey
of Heilbronn \cite{H67}.

\begin{lemma}\label{dedekind}  Let $A>0$ and let
$f\in \ZZ[x]$ be an irreducible cubic polynomial with content $1$. 
Then we have
$$
\sum_{n\leq X}\frac{\chi
(n)\rho_f(n)}{n}=\vartheta(f;\chi)+O_{A}\big((\log X)^{-A}\big),
$$
with $\vartheta(f;\chi)=L(1,\chi)H_f(1,\chi)$. Furthermore, we have 
$$
\sum_{p\leq X}\frac{\chi (p)\rho_f(p)}{p} \ll  1.
$$
\end{lemma}

In the present investigation we will be concerned with the case
$f(x)=C(x,1)$, an irreducible polynomial of degree $d=3$ 
defined over $\ZZ$. 
We will need to relate the series
\begin{equation}
  \label{eq:defH}
D(s)=\sum_{n=1}^\infty \frac{\chi(n)\rho(1,n)}{n^{1+s}}
\end{equation}
to $H_{C(x,1)}(s)$, where $\rho(d_1,d_2)$ is given by
\eqref{defrho}. To this end it will be necessary to have some 
further information about the size of  $\rho(d_1,d_2)$ at prime powers. 
We will suppose once and for all that 
\begin{equation}
  \label{eq:LC}
L (\x)=a x_1+b x_2,\quad
C(\x) =c_0x_1^3+c_1x_1^2x_2 +c_2x_1 x_2^2+c_3x_2^3,
\end{equation}
for $a,b,c_i\in \ZZ$, with non-zero integers
\begin{equation}\label{defDelta}
\Delta=|\Res (L,C)|, \quad \Delta'=|\disc(C)|.
\end{equation}
Our investigation is summarised in the
following result.

\begin{lemma}\label{rhopremier}
Let $C\in \ZZ[\x]$ be an irreducible cubic form and let $L\in
\ZZ[\x]$ be a non-zero linear form. 
Assume that $L,C $ are primitive
 and let $\Delta$, $\Delta'$
be as in \eqref{defDelta}. Then we have the following expressions.
\begin{enumerate}
\item
When $p\nmid  c_0 \Delta'$  and $\nu\in\NN$ then  we have
$$
\rho (1,p^\nu)= 
\begin{cases}
p^{\nu-1}(p^{[\frac{\nu}{3}]}-1)\rho_{C(x,1)}(p)+p^{\nu+[\frac{\nu}{3}]}
&\mbox{if  $\nu\equiv 0\bmod 3$},\\  
p^{\nu-1
}(p^{[\frac{\nu}{3}]+1}-1)\rho_{C(x,1)}(p)+p^{\nu+[\frac{\nu}{3}]-1}
&\mbox{if  $\nu\equiv 1\bmod 3$}, 
\\ 
p^{\nu-1
}(p^{[\frac{\nu}{3}]+1}-1)\rho_{C(x,1)}(p)+p^{\nu+[\frac{\nu}{3}]}&\mbox{if
  $\nu\equiv 2\bmod 3$}. 
\end{cases}
$$
In particular, when $p\nmid  c_0 \Delta'$ we have 
$$
\rho (1,p)=(p-1)\rho_{C(x,1)}(p)+1.
$$
For any prime $p$ and $\nu\in\NN$, we have  
$$
\rho(1,p^{\nu }) \ll \min\{p^{  2\nu -1},p^{ \frac{4\nu}{3}}\}.
$$
\item 
When $\nu_2\leq 3\nu_1$ and $p\nmid \Delta $, we have
$$
\rho(p^{\nu_1},p^{\nu_2}) 
 \leq p^{ \nu_1+2\nu_2-\lceil\frac{\nu_2}{3}\rceil} .
$$
When $0\leq 3\nu_1<\nu_2$ and $p\nmid c_0\Delta\Delta'$,
we have
$$
\rho(p^{\nu_1},p^{\nu_2}) 
\leq \Big(3+\frac{1}{p}\Big) p^{ 2\nu_1+ \nu_2+[\frac{\nu_2}{3}]}.
$$
\item 
For any prime $p$ and $\nu_1,\nu_2 \in\ZZ_{\geq 0}$ we have
$$
\rho(p^{\nu_1},p^{\nu_2}) \ll \min\{
p^{\nu_1+2\nu_2},p^{2\nu_1+2\nu_2-1},p^{2\nu_1+\frac{4\nu_2}{3}}\}.
$$
\end{enumerate}
\end{lemma}

\begin{proof} 
These expressions are founded on a preliminary study of the related quantity 
\begin{equation}
  \label{eq:star}
\rho^*(p^{\nu_1},p^{\nu_2})
=\# 
\{ \x \in \ZZ^2\cap [0,p^{\nu_1+\nu_2})^2 : p^{\nu_1} \mid L (\x), ~
p^{\nu_2}\mid C(\x), ~p\nmid \x\}.
\end{equation}
We will follow the convention that $\rho^*(1,1)=1$. 
We can relate this quantity to $\rho(p^{\nu_1},p^{\nu_2})$ via the
easily checked identity
\begin{equation}
  \label{eq:r1}
\rho(p^{\nu_1},p^{\nu_2})
 = \sum_{0\leq k\leq \max\{\nu_1,\lceil
 \frac{\nu_2}{3}\rceil \}} 
 \rho^*\big(p^{\max\{\nu_1-k,0\}}, p^{\max\{\nu_2-3k,0\}}\big)
 p^{m_k},  
\end{equation}
with $m_k=2(\min\{\nu_1,k\}+ \min\{\nu_2,3k\}-k)$.
This follows on partitioning the $\x$ to be counted according to the
common $p$-adic order of $x_1,x_2$ and $p^{\max\{\nu_1,\lceil
 \frac{\nu_2}{3}\rceil \}}$.

Proceeding with our  analysis of $\rho^*(p^{\nu_1},p^{\nu_2})$, we
begin by noting that 
\begin{equation}
  \label{eq:sun}
\rho^*(1,p^\nu)=\varphi(p^\nu) \rho_{C(x,1)}(p^\nu)
\end{equation}
if $p\nmid  c_0$,
since the solutions $\x$ to be counted satisfy
$p\nmid x_2$ for $p\nmid c_0$. 
Hence Lemma \ref{dan2} yields
$
\rho^*(1,p^\nu)\leq 3\varphi(p^\nu)
$
if $p\nmid  c_0 \Delta'$. 
Suppose now that $p\mid c_0\Delta'$. 
If $\x$ is counted
by $\rho^*(1,p^\nu)$ then $\xi\leq v_p(c_0)$ if $p^{\xi}\| x_2$.
We may conclude from Lemma \ref{dan2} that 
\begin{equation}
  \label{eq:moon}
\rho^*(1,p^\nu)\leq \sum_{0\leq \xi \leq v_p(c_0)}
\phi(p^{\nu-\xi})\cdot p^\xi \rho_{p^{-\xi}C(x,p^\xi)}(p^{\nu-\xi})  
\ll p^\nu,
\end{equation}
where we recall our convention that the implied constants are allowed
to depend on the coefficients of $L,C$. 
This latter estimate holds for any prime $p$.
Next we note that 
$$
\rho^*(p^{\nu_1},p^{\nu_2})\leq \min\{p^{2\nu_2} 
\rho^*(p^{\nu_1},1),p^{2\nu_1} 
\rho^*(1,p^{\nu_2})\}. 
$$
Since $\rho^*(p^{\nu_1},p^{\nu_2})=0$
when $\min\{\nu_1,\nu_2\}>v_p(\Delta )$, and 
$\rho^*(p^{\nu_1},1)=\phi(p^{\nu_1})$,
it therefore follows from \eqref{eq:moon} that
\begin{equation}
  \label{eq:linden}
  \rho^*(p^{\nu_1},p^{\nu_2})\ll p^{\nu_1+\nu_2}.
\end{equation}

We are now ready to deduce the statement of Lemma \ref{rhopremier}.
When $p \nmid \Delta'$ and $\nu\geq 1$ it follows from Hensel's lemma
that $\rho_{C(x,1)}(p^\nu)=\rho_{C(x,1)}(p)$. 
The first pair of displayed relations in part (1) now follow 
directly from \eqref{eq:r1} and \eqref{eq:sun}.
The final part is again based on \eqref{eq:r1}, but now combined
with \eqref{eq:moon}.

Turning to the proof of 
part (2), for which we call upon \eqref{eq:r1}, we see
that when $\nu_2\leq 3\nu_1$
and $p\nmid \Delta$ we have 
\begin{align*}
\rho(p^{\nu_1},p^{\nu_2})=
\sum_{\lceil\frac{\nu_2}{3}\rceil\leq k\leq \nu_1}p^{2\nu_2}\rho^*(p^{\nu_1-k},1)
&=p^{2\nu_2}\sum_{\lceil\frac{\nu_2}{3}\rceil\leq k\leq
  \nu_1}\phi(p^{\nu_1-k})
\leq
p^{ \nu_1+2\nu_2-\lceil\frac{\nu_2}{3}\rceil} .
\end{align*} 
When  $3\nu_1<\nu_2$ and $p\nmid
c_0\Delta\Delta'$ we have 
\begin{align*}
\rho(p^{\nu_1},p^{\nu_2})&=
\sum_{\nu_1\leq k\leq [\frac{\nu_2}{3}]}p^{2\nu_1+4k}\rho^*(1,p^{\nu_2-3k})
+\Big(\Big\lceil\frac{\nu_2}{3}\Big\rceil-\Big[\frac{\nu_2}{3}\Big]\Big)p^{2\nu_1+2\nu_2-2\lceil\frac{\nu_2}{3}\rceil}
\cr
&\leq 3
p^{ 2\nu_1+ \nu_2+[\frac{\nu_2}{3}]}     
+\Big(\Big\lceil\frac{\nu_2}{3}\Big\rceil- 
\Big[\frac{\nu_2}{3}\Big]\Big)p^{2\nu_1+2\nu_2-2\lceil\frac{\nu_2}{3}\rceil}  
\cr
&\leq \Big(3+\frac{1}{p}\Big) p^{ 2\nu_1+ \nu_2+[\frac{\nu_2}{3}]}.
\end{align*} 
Finally part (3) is a consequence of the inequalities
$$
\rho(p^{\nu_1},p^{\nu_2})\leq p^{2\nu_2}\rho( p^{\nu_1},1)=p^{\nu_1+2\nu_2},\qquad
\rho(p^{\nu_1},p^{\nu_2})\leq p^{2\nu_1}\rho(1,p^{\nu_2}),
$$
together with part (1) of the lemma.
\end{proof}

In general the forms $L,C$ need not be primitive. 
We let $\ell_1,\ell_2 \in \NN$ and $L^*, C^*  $ be primitive forms
such that
$$
L =\ell_1L^*,\quad 
C=\ell_2C^*.
$$
One can easily restrict attention to primitive forms in Lemma
\ref{rhopremier} via the trivial observation that 
\begin{equation}
  \label{eq:mars}
  \frac{\rho(\bd;L,C)}{(d_1 d_2)^2}
=\frac{\rho(\bd';L^*,C^*)}{(d_1' d_2')^2} ,
\end{equation}
for any $\bd\in\NN^2$, where $d_i'= \gcd(d_i,\ell_i)^{-1}d_i$.

Returning to 
the Dirichlet series $D(s)$ defined in \eqref{eq:defH}, we write
\begin{equation}\label{defG}
D(s)=
G_{C(x,1)}(s,\chi)A(s),
\end{equation}
where $G_{C(x,1)}(s,\chi)$ is given by \eqref{defHf} and $A(s)$ is
the Dirichlet series associated to an appropriate arithmetic
function $a$. We will need the following result.

\begin{lemma}\label{H} For any $\ve>0$ and $\sigma\geq
  \frac{5}{6}+\ve$ we have
$
\sum_{n=1}^\infty |a(n)|n^{-\sigma}\ll 1.
$
\end{lemma}

\begin{proof}
Since the two functions involved are multiplicative 
it suffices to analyse the Euler products
$$
D(s)=\prod_pD_p(s),\quad
G_{C(x,1)}(s,\chi)=\prod_p G_{p,C(x,1)}(s,\chi).
$$
Suppose that $\Re e
(s)=\sigma>\frac{2}{3}$.  
When $p\nmid c_0\Delta'$, 
Lemma \ref{dan2} and part (1) of Lemma \ref{rhopremier} yield 
\begin{align*}
D_p(s)&=1+\frac{\chi(p)\rho_{C(x,1)}(
p)}{p^{s}}+O\big(p^{-2\sigma+\frac{2}{3}}+p^{-1-\sigma}\big)
\cr
&=G_{p,C(x,1)}(s,\chi)\Big(1+O\big(p^{-2\sigma+\frac{2}{3}}+p^{-1-\sigma}+p^{-2\sigma}
\big)\Big).
\end{align*}
When $p\mid c_0\Delta'$, we have
$$
D_p(s)=1+  O\big(p^{\frac{2}{3}- \sigma}\big)
,\quad G_{p,C(x,1)}(s,\chi)=1+  O\big(p^{\frac{2}{3} - \sigma}\big).
$$

From this we deduce that \eqref{defG} holds with the Dirichlet
series $A$ associated to a function $a$ satisfying the bound recorded in
the lemma.
\end{proof}

We close this section with a simple result 
concerning the estimation of summatory functions that involve the
convolution of arithmetic functions.  

 \begin{lemma}\label{convol} Let $ A>0$. Let $g, h$ be arithmetic
functions and $C,C',C''$ constants such that
$$
\sum_{d=1}^\infty \frac{|h(d)|(\log 2d)^A}{ d}\leq C'',\quad 
\sum_{d\leq x} \frac{g(d)}{ d}=C +O\Big(\frac{ C' }{(\log
  2x)^{A}}\Big).
$$ 
Then we have 
$$
\sum_{n\leq x}\frac{(g*h)(n)}{n}=C \sum_{d=1}^\infty \frac{ 
h(d)}{ d}+O\Big(\frac{C''(C+C')}{(\log 2x)^{A}}\Big).
$$
\end{lemma}

\begin{proof}
We clearly have
$$
\sum_{n\leq x}\frac{(g*h)(n)}{n}=\sum_{d\leq x} \frac{ 
h(d)}{ d}\sum_{m\leq \frac{x}{d}} \frac{g(m)}{ m}.
$$
We approximate the inner sum over $m$ by $C$ if $d\leq \sqrt{x}$. On
noting that
$$ \sum_{d>\sqrt{x}} \frac{| h(d)|  }{ d}
\leq \sum_{d=1}^\infty \frac{|h(d)|}{ d}\frac{(\log 2d)^A}{(\log
2\sqrt{x})^{A}}\ll
\frac{C''}{ (\log 2x)^{A}},
$$
we are easily led to the conclusion of the lemma.
\end{proof}

\section{Preliminary steps}

In this section we shall begin the proof of Theorem \ref{mainS}. 
Recall the notation \eqref{eq:LC} and \eqref{defDelta} concerning $L,C$.
We will find it convenient to estimate the corresponding sum
$
S_0(X)$, say, in which we insist that the greatest common divisor of
$x_1,x_2$ is 
odd. Note that $r(2n)=r(n)$ for any positive integer $n$. 
We may therefore write 
$$
S(X)=\sum_{k_0\geq 0} \sum_{\substack{\x\in\ZZ^2\cap X\mcal{R}\\ 2^{k_0}
    \| \x}} 
r(L(\x)) r(C(\x))=
\sum_{k_0\geq 0} S_0(2^{-k_0}X).
$$
We will also need to extract $2$-adic factors from 
$L(\x)$ and $C(\x)$. Thus we have
$$
S(X)=\sum_{k_0\geq 0}\sum_{\ma{k}=(k_1,k_2)\in\ZZ_{\geq
0}^2}S_{\ma{k}}(2^{-k_0}X),
$$
where $S_{\ma{k}}(X)$ is the restriction of $S(X)$ to $\x$ for which
$2^{-k_1}L(\x)\equiv 1\bmod 4$ and  $2^{-k_2}C(\x)\equiv 1\bmod 4$,
with $2\nmid \x$. 
In particular it is clear that $k_1,k_2\ll \log X$ 
and $\min \{k_1,k_2\}\leq v_2(\Delta)$  
in order for 
$S_{\ma{k}}(2^{-k_0}X)$ to be non-zero.
We will need to show that the available range for $k_1,k_2$ can be
reduced with an acceptable error. A straightforward application of
\cite[Corollary 1]{nair} yields 
$$
S_{\ma{k}}(X)\ll
2^{\ve(k_1+k_2)}(2^{-\max\{k_1,k_2\}}X^2+X^{1+\ve}),
$$ 
for any
$\ve>0$. It follows that
\begin{equation}\label{SXsuma}
S(X)=\sum_{k_0\geq 0}\sum_{0\leq k_1,k_2\leq \log\log X}
S_{\ma{k}}(2^{-k_0}X) +O\big(X^2(\log X)^{-(1-\ve)\log 2}\big).
\end{equation}

The condition $2^{-k_1}L(\x)\equiv 1\bmod 4$ is easy to
analyse. Without loss of generality we may assume that $a$ is odd. Let
$0\leq c<2^{k_1+2}$ be such that $ac\equiv -b\bmod {2^{k_1+2}}$ and
$c'\in\{-1,1\}$ such that $c'\equiv a\bmod 4$. Then we see that 
$2^{-k_1}L(\x)\equiv 1\bmod 4$ is equivalent to the
existence of $x_1'\equiv 1\bmod 4$ such that
$$
x_1=cx_2+c'2^{k_1}x_1'.
$$ 
If $k_1\geq 1$, the condition that $2\nmid \x$ reduces to the
condition that  $x_2$ should be odd. If $k_1=0$, the condition
$2\nmid \x$ holds automatically.
 
Next we note that the condition $2^{-k_2}C(\x)\equiv 1\bmod 4$ can be written
$$
C(cx_2+c'2^{k_1}x_1', x_2)\equiv 2^{k_2}{x_1'}^3 \bmod{2^{k_2+2}}.
$$
If the form 
$C(cY+c'2^{k_1}X, Y)$ has all coefficients divisible by $2^{k_2+1}$
then this congruence has no solutions. Otherwise define $k_1'\leq k_2$
so that $2^{k_1'}$ is the largest power of $2$ dividing all the
coefficients, and set 
$C(cY+c'2^{k_1}X,Y)=2^{k_1'}C_0(X,Y)$. Writing 
$k_2'=k_2-k_1'\geq 0$ then we see that the above congruence is
equivalent to  
$
C_0(x_1',x_2)\equiv 2^{k_2'}{x_1'}^3\bmod{2^{k_2'+2}}.
$
Since $x_1'$ is odd we have
$
x_2 \equiv \alpha x_1'\bmod{2^{k_2'+2}},
$
for $\alpha\in [0,2^{k_2'+2})$ being one of the roots of 
\begin{equation}
  \label{eq:comet}
C_0(1,\alpha)\equiv
2^{k_2'}
\bmod{ 2^{k_2'+2}}.
\end{equation}
The condition that $x_2$ be odd,
which should be added when $k_1\geq 1$, is therefore equivalent to the
condition that $\alpha$ be odd. 
Finally we make the change of variables
$x_2=\alpha x_1'+2^{k_2'+2}x_2'$ and note
that $x_1',x_2'\ll X$ whenever $\x\in X\mcal{R}.$
We denote by $n(k_1,k_2)$ the number
of available $\alpha$ and recall from above that 
$\min\{k_1,k_2\}\leq v_2(\Delta)$.
Since $a$ is odd we clearly have 
$$
n(k_1,k_2)\ll \#\{x \bmod{2^{k_1+k_2}}: ~x\equiv -ba^{-1}
\bmod{2^{k_1}}, ~C(x,1)\equiv 0 \bmod{2^{k_2}}\}.
$$
If $k_2\leq k_1$ then the right hand side is at most $2^{k_2}
\ll 1$. 
If $k_2>k_1$ then the right hand side is at
most $2^{k_1}\rho_{C(x,1)}(2^{k_2})\ll 1$ 
by Lemma \ref{dan2}. Hence we have 
\begin{equation}
  \label{eq:AV}
  n(k_1,k_2)\ll 1.
\end{equation}

In summary we have shown that the conditions 
$v_2(L(\x))=k_1$, $v_2(C(\x))=k_2$ and $2\nmid \x$,
with $2^{-k_1}L(\x)\equiv 1\bmod 4$ and $2^{-k_2}C(\x)\equiv 1\bmod 4$,
can be written $\x=\M \x'$ with $x_1'\equiv 1\bmod 4 $ and 
$$
\M=\M_\alpha=
\Big(\begin{array}{cc} c'2^{k_1}&c\\ 0 &1\end{array}\Big)
\Big(\begin{array}{cc} 1&0\\ \alpha &2^{k_2'+2}\end{array}\Big)
=
\Big(\begin{array}{cc} c'2^{k_1}+c\alpha&c2^{k_2'+2}\\ \alpha
  &2^{k_2'+2}\end{array}\Big),$$
where $\alpha$ is a zero of \eqref{eq:comet} that should be odd when
$k_1\geq 1$. We note that 
\begin{equation}\label{eq:detM}
|\det \M|=2^{k_1+k_2'+2}.
\end{equation}
Furthermore,  a little thought reveals that 
\begin{equation}\label{calculK2}
K_2=\sum_{k_0\geq 0}\frac{1}{2^{2k_0}}\sum_{k_1,k_2\geq 0}\frac{ n(k_1,k_2)
}{2^{k_1+k_2'+2}}=
\frac{1}{3}\sum_{k_1,k_2\geq 0}\frac{ n(k_1,k_2)
}{2^{k_1+k_2'}},
\end{equation}
in the notation of Theorem \ref{mainS}.

We are now ready to start our analysis of $S(X)$   in earnest, for
which we follow the line of attack in \cite{4linear} and
\cite{h-b03}. 
 In the present
investigation we will not seek complete uniformity in $L,C$ and
$\mcal{R}$, unlike in \cite{4linear}, which will greatly streamline our 
exposition. 
Let us set $Y= X^{\frac{1}{2}}(\log X)^{-C}$ with $C$ a large
unspecified constant. 
When $0<n\ll X^3$  
and $n'=2^{-v_2(n)}n\equiv 1\bmod{4}$, we write
\begin{align*}
r(n)=r(n')
&=4\sum_{\substack{d_2\mid n'\\ d_2\leq {X}^{\frac{3}{2}}}}\chi (d_2)+
4\sum_{\substack{e_2\mid n'\\e_2> {X}^{\frac{3}{2}}}}\chi (e_2)\cr
&=4\sum_{\substack{d_2\mid n\\ d_2\leq {X}^{\frac{3}{2}}}}\chi (d_2)+
4\sum_{\substack{d_2\mid n\\ n'>d_2 {X}^{\frac{3}{2}}}}\chi (d_2)\\
&=4A_+(n)+4A_-(n).
\end{align*}
We will apply this with $n=C(\x)$. In the same manner 
when $0<m\ll X$  
we can write 
$$
r(m)  =4B_+(m)+4B_0(m)+4B_-(m),
$$
under the hypothesis that $m'=2^{-v_2(m)}m\equiv 1\bmod{4}$, with 
$$
B_+(m)=\sum_{\substack{d_1\mid m\\ 
d_1\leq Y}}\chi (d_1),\quad 
B_0(m)=
\sum_{\substack{d_1\mid
m\\ Y<d_1\leq \frac{X}{Y}}} \chi (d_1),
\quad
B_-(m)=\sum_{\substack{d_1\mid m\\ 
m'>d_1 \frac{X}{Y}}}\chi (d_1).
$$
Making the transformation $\x=\M \x'$, 
it follows that
$$
S_{\ma{k}}(X)=\sum_\alpha S_{\ma{k},\alpha}(X),
$$
where
$$
S_{\ma{k},\alpha}(X)=\sum_{\substack{\x'\in \ZZ^2\cap
X\mcal{R}_\M \\  x_1'\equiv 1\bmod 4}}r(L_\M(\x'))r(C_\M(\x')),
$$
with 
$$
\mcal{R}_\M =\{ \x'\in\RR^2\,:\, \M\x'\in\mcal{R}\}, \quad
L_\M(\x')= L(\M\x'),\quad
C_\M(\x')= C(\M\x').
$$
The region $\mcal{R}_\M$ has volume
$2^{-k_1-k_2'-2}\vol(\mcal{R})$ and  is contained in a box with side
length $\ll |\det \M|^{-1}2^{k_1+k_2'}\ll 1$. 
Collecting together the above we may conclude that
\begin{equation}\label{defSX}
S_{\ma{k}}(X)=16\sum_{\al}
\sum_{\pm,\pm}S_{\pm,\pm}(X;\ma{k},\alpha)+4T(X;\ma{k},\alpha), 
\end{equation}
with
\begin{equation}
  \label{eq:Spm}
S_{\pm,\pm}(X;\ma{k},\alpha)=\sum_{\substack{\x'\in \ZZ^2\cap
X\mcal{R}_\M \\  x_1'\equiv 1\bmod 4}}A_\pm(C_\M(\x'))B_\pm(L_\M(\x'))
\end{equation}
and
$$
T(X;\ma{k},\alpha)=\sum_{\substack{\x'\in \ZZ^2\cap
X\mcal{R}_\M \\  x_1'\equiv 1\bmod 4}}r(C_\M(\x'))B_0(L_\M(\x')).
$$

The sums $S_{\pm,\pm}(2^{-k_0}X;\ma{k},\alpha)$ will make up the main term in
our final asymptotic formula and we save their analysis for the
following section. We dedicate the remainder of this section to
showing that $T(2^{-{k_0}}X;\ma{k},\alpha)$ makes a satisfactory overall
contribution 
$$
\sum_{k_0\geq 0} \sum_{0\leq k_1,k_2\leq \log\log X}
\sum_\alpha
T(2^{-k_0}X;\ma{k},\alpha)=T(X), 
$$
say, to the error term. 
By \eqref{eq:AV} we have
$$ 
T(X)\ll (\log \log X)^2\sum_{k_0\geq 0} 
\sum_{m \in \mathcal{B}(k_0)}T_m(2^{-k_0}X)|B_0(m)|,
$$
where $\mathcal{B}(k_0)$ is defined to be the intersection
$$ 
\{ m\in\ZZ: \mbox{$\exists d\mid m$ s.t.\ $Y<d\leq
  XY^{-1}$}\} \cap 
\{m\in \ZZ:\mbox{$\exists
\x\in 2^{-k_0}X\mcal{R}$ s.t.\ $L (\ma{x})=m$}\}
$$
and 
$$
T_m(X)=\sum_{\substack{\x\in \ZZ^2\cap X\mcal{R}\\ L(\x)=m}} r(C(\x)).
$$
But then \cite[Lemma 6]{4linear} yields
$$ T(X)\ll  X \frac{(\log\log X)^{\frac{17}{4}}}{(\log X)^{\eta}} 
\sum_{k_0\geq 0}2^{-k_0}
\max_{\substack{m\in \NN}} |T_m(2^{-k_0}X)|,
$$ 
where $\eta=1-\frac{1+\log\log 2}{\log 2}$.
Once combined with the following result this is therefore enough to
conclude the proof that $T(X)\ll X^2(\log X)^{-\eta+\ve}$, which
suffices for Theorem \ref{mainS}.

\begin{lemma}
Let $\ve>0$ and let $m\leq X$. Then we have 
$$
T_m(X)\ll X (\log X)^\ve.
$$
\end{lemma}

\begin{proof}
We consider here the case $a\neq 0$, the case $b\neq 0$ being dealt
with similarly. The relation $L(\x)=m$ allows us to
write $x_1=a^{-1}(m-bx_2)$ and
$$
C(\x)=\frac{1}{a^3}C(m-bx_2,ax_2)=\frac{1}{a^3}(c_3'x_2^3+c_2'mx_2^2+c_1'm^2x_2 + c_0'm^3),
$$
with 
$$
c_3'=C( -b ,a),\quad 
c_2'=3b^2c_0-2abc_1+a^2c_2,\quad 
c_1'=-3bc_0+c_1a,\quad 
c_0'=c_0.
$$
Let $\delta_m=\gcd_{0\leq i\leq 3}(c_i'm^{3-i})$, so that
$C_m(x_2)= a^3\delta_m^{-1}C(\x)$ is primitive as a polynomial in $x_2$.
It follows that
$$
T_m(X)
\leq 
\sum_{\substack{\x\in \ZZ^2\cap
X\mcal{R}\\ 
L(\x)=m}} 
r(a^4C(\x))\leq\sum_{x_2\ll X}  r(a\delta_mC_m(x_2)).
$$
The rest of the proof has much in common with the proof of \cite[Lemma
5]{4linear} and so we shall attempt to be brief. 

Write $r_0(n)=\frac{1}{4}r(n)$ and $r_1$ for the multiplicative function
defined via
$$
r_1(p^\nu)=
\begin{cases} 
\nu+1, & \mbox{if $p\mid 3a\delta_m$},\\
r_0(p^\nu) , & \mbox{otherwise}.
\end{cases}
$$
We obtain
$$
T_m(X)\leq 4 \tau(a\delta_m) \sum_{x_2\ll X}  r_1(C_m(x_2)).
$$
Clearly $\delta_m\mid c_3'\neq 0$, whence 
$\tau(a\delta_m)\ll 1.$
The polynomial $C_m\in \ZZ[x_2]$ has degree $3$ and is both primitive
and irreducible over $\QQ$. Therefore the only possible fixed prime divisors are
$2$ and $3$.  An application of \cite[Lemma 5]{nair} allows one to
deduce that there exists $\alpha\mid 36$, $m_2,m_3\leq 9$ and 
$\gamma=2^{m_2}3^{m_3}$ such that the polynomial
$$
g_{\alpha,\beta}(x_2)=\frac{C_m(\alpha x_2+\beta)}{\gamma}
$$ 
is without any fixed prime divisor for each $\beta$ modulo $\alpha$.  We obtain
$$ 
\sum_{x_2\ll X}  r_1(C_m(x_2))\ll \sum_\alpha\sum_{\beta\bmod \alpha}
\sum_{x_2\ll  X}
r_1(g_{\alpha,\beta}(x_2)).
$$
Since $\|g_{\alpha,\beta}\|\ll \|C_m\| 
\ll m^3$, it now follows from \cite[Theorem 2]{nair} that
$$
\sum_{x_2\ll X}  r_1(C_m(x_2))\ll X
\sum_\alpha\sum_{\beta\bmod \alpha}
\prod_{p\ll X}
\Big\{\Big(1-\frac{\rho_{g_{\alpha,\beta}}(p)}{p}\Big)
\sum_{\nu\geq 0}\frac{\rho_{g_{\alpha,\beta}}(p^{\nu})r_1(p^{\nu})}{p^{\nu}}\Big\},
$$
because $X\gg m^{\varepsilon},$ where $\rho_{g_{\alpha,\beta}}(p)$ is
given by \eqref{eq:f}.
A straightforward consideration of discriminants (see \cite[Lemma
1]{nair}, for example) yields
$\disc(g_{\alpha,\beta})\ll m^6$.
 
To go further it is clear that we will need good upper bounds for the
function $\rho_{g_{\alpha,\beta}}(p^{\nu})$ for prime powers
$p^\nu$. Such estimates are furnished by Lemma \ref{dan2}.
Thus for any prime $p$ we deduce that
$$
\sum_{\nu\geq 1} \frac{\rho_{g_{\alpha,\beta}}(p^{\nu})r_1(p^{\nu})}{p^{\nu}} 
\ll \frac{1}{p}.
$$
By including a factor 
$$
\ll \prod_{p\mid \disc(g_{\alpha,\beta})}\Big(1+\frac{1}{p}\Big)^{O(1)}
\ll (\log\log m)^{O(1)} \ll (\log X)^\ve,
$$ 
we take care of the primes $p\mid \disc(g_{\alpha,\beta})$. 
Next, for any $p\nmid \disc(g_{\alpha,\beta})$, we have 
$$
\sum_{\nu\geq 2}
\frac{\rho_{g_{\alpha,\beta}}(p^{\nu})r_1(p^{\nu})}{p^{\nu}} 
\ll \frac{1}{p^2},
$$
which allows us to ignore the exponents $\nu\geq 2$.

For any prime $p\geq 5$, we have $\rho_{g_{\alpha,\beta}}(p
)=\rho_{C_m}(p )$, which for $p\nmid ac_3'$ is equal to
$\rho_{C(m-bx_2,a x_2)}(p ).$
If $p\nmid ma$ then the map $\ZZ/p\ZZ\setminus \{ mb^{-1}\}\rightarrow
\ZZ/p\ZZ$, given by $x_2\mapsto ax_2(m-bx_2)^{-1}$ is injective. It follows that
$
\rho_{g_{\alpha,\beta}}(p )=\rho_{C(1,x)}(p ),
$
for $p\geq 5$ and $p\nmid ma c_3'$. Observing that
$r_0(p)=1+\chi(p)$, our investigation so far has
therefore shown that 
\begin{align*}
\sum_{x_2\ll X} r_1\big(g_{\al,\be}(x_2)\big)
&\ll X(\log X)^\ve
\prod_{\substack{p\ll X\\ p\nmid \disc(g_{\alpha,\beta})}} \Big(1+\frac{\rho_{C(1,x)}(p)(r_0(p)-1)}{p}\Big) \\
&\ll X(\log X)^\ve
 \prod_{p\ll X} \Big(1+\frac{\chi(p)\rho_{C(1,x)}(p)}{p}\Big)\\
&\ll X(\log X)^\ve,
\end{align*}
by Lemma \ref{dedekind}.   This
therefore completes the proof of the lemma.
\end{proof}

\section{Level of distribution}

The focus of this section is upon estimating the sums in
 \eqref{eq:Spm}.  For any $\bd\in \NN^2$ let
$$
\mathsf{\Lambda}({\ma{d}})=\mathsf{\Lambda}({\ma{d}};L,C)=  \{ \x \in
\ZZ^2 : d_1 \mid L (\x), ~ d_2\mid C(\x)\}
$$
and let 
$\mathsf{\Lambda}_{\M}({\ma{d}})=\mathsf{\Lambda}({\ma{d}};L_\M,C_\M)$.
Given any region $\mcal{A}\subset \RR^2$, we will write  
$X\mcal{A}_4$ for the set $\{\x\in \ZZ^2\cap X\mcal{A}: x_1\equiv 1
\bmod{4}\}$.  
We clearly have 
$$ 
S_{\pm,\pm}(X;\ma{k},\alpha)
=\sum_{\substack{d_1\ll Y \\ d_2\ll {X}^{\frac{3}{2}}}}\chi(d_1d_2)\#(
\mathsf{\Lambda}_\M({\ma{d}})\cap X\mcal{R}_4^{\pm,\pm}(\ma{d},\M)),
$$
with,  for example, 
$$
X\mcal{R}^{-,-}(\ma{d},\M)=
\{\x'\in X\mcal{R}_\M :  C_\M(\x')>d_2 {X}^{\frac{3}{2}}, ~L_\M(\x')>d_1
XY^{-1}\}.
$$

Let $\|\M\|$ denotes the maximum modulus of any entry in the matrix
$\M$ and let 
$\rho_\M(\bd) =\rho(\bd;L_\M,C_\M)$, in the notation of \eqref{defrho}.
Loosely speaking the 
idea is now to rewrite the inner cardinality as a sum of 
cardinalities, each one over lattice points belonging to an
appropriate region. We would like to approximate each such cardinality
by its  volume. In doing so we need to show that the associated error term makes 
a satisfactory overall contribution once summed over the remaining
parameters.  This is the essential content of the 
following ``level of distribution'' result.

\begin{lemma}\label{LOD} 
Let $\varepsilon>0$ and let $V_1, V_2,X  \ge
2$. 
Assume that $C\in \ZZ[\x]$ is an irreducible cubic form and let $L\in
\ZZ[\x]$ be a non-zero linear form. 
Then there exists an absolute constant $A>0$ such that
\begin{align*}
\sum_{\substack{\ma{d}\in\NN^2 \\ d_i\le V_i\\ 2\nmid d_1d_2}}
 \sup_{ \partial(\mcal{A})\leq M}
\Bigg|\#\big(\mathsf{\Lambda}_\M({\ma{d}})\cap X\mcal{A}_4 \big)
&  -  \frac{\vol(\mcal{A} )X^2\rho_\M(\bd) }{4(d_1d_2 )^2} \Bigg|\\
&\ll
\|\M\|^\ve  (MX(\sqrt{V_1V_2}+V_1)+V_1V_2)(\log V_1V_2)^{A},
\end{align*}
where
the supremum  is taken over 
compact subsets $\mcal{A}\subset \RR^2$ 
whose boundary is a
piecewise continuously
differentiable closed curve with length $\partial(\mcal{A})\leq M$ and 
throughout which $L(\x)>0$ and $C(\x)>0$.
\end{lemma}

We will not prove this result here, following closely as it does the
arguments developed in \cite[Lemme 5]{L1L2Q}, \cite[Lemma 3.2]{D99} and 
\cite[Proposition 1]{M08}.  
Now it follows from \eqref{eq:detM} that $d_1d_2$ is coprime to $\det
\M$, so that $\rho_\M(\ma{d})=\rho(\ma{d};L ,C)=\rho(\ma{d} ).$
We may therefore conclude from Lemma~\ref{LOD} that
$$ 
S_{\pm,\pm}(X;\ma{k},\alpha)
=\sum_{\substack{d_1\ll Y \\ d_2\ll {X}^{\frac{3}{2}}}}\frac{ 
\chi(d_1d_2)
 \vol(\mcal{R}^{\pm,\pm}(\ma{d},\M) ) X^2\rho(\bd) }{4(d_1d_2
 )^2}+O\Big(\frac{2^{\ve(k_1+k_2)}X^2}{(\log X)^{ \frac{C}{2}-A}}\Big). 
$$
Choosing $C=2A+8$ and replacing $X$ by $2^{-k_0}X$, we see that the overall contribution 
from this error term is 
$$
\ll \sum_{k_0\geq 0}
\frac{(2^{-k_0}X)^2}{(\log X)^4} \sum_{k_1,k_2\leq 
\log \log X}2^{\ve(k_1+k_2)} n(k_1,k_2)\ll \frac{X^2}{(\log X)^2},
$$
by \eqref{eq:AV}. This is satisfactory for Theorem \ref{mainS}.

Our final task is to produce an asymptotic formula for the sum
$$
S(V_1,V_2)=\sum_{\substack{\bd \in\NN^2\\ 
d_i\leq V_i}}
  \frac{\chi(d_1d_2)\rho(\bd) }{(d_1d_2 )^2}.
$$
Recall the definition of $K_p$ from the statement of Theorem \ref{mainS}.
We will establish the following result.

\begin{lemma}\label{SV1V2} Let $\varepsilon>0$ and $A>0$.
For any $V_1,V_2\geq 2$ we have 
$$
S(V_1,V_2)=\frac{\pi^2}{16} K' 
+O\Big(\frac{\log V_{\min}}{(\log V_{\max})^{A}} + \frac{1}{(\log V_{\min})^{A}} \Big) 
$$
where $V_{\min}=\min \{V_1,V_2\}$, $V_{\max}=\max \{V_1,V_2\}$ and 
$
K'=\prod_{p>2} K_p.
$
\end{lemma}

\begin{proof} 
We begin by establishing the lemma for the case in which 
$L$ and $C$ are both primitive.
We first consider the case  $V_1\geq V_2$.
The sum to be estimated can be written
$$
S(V_1,V_2)
=\sum_{ d_2\leq V_2} \frac{\chi( d_2)\rho(1,d_2) }{ d_2^2}S_1
(V_1,d_2),
$$
with
$$
S_1 (V_1,d_2)=
 \sum_{ d_1\leq V_1} \frac{\chi(d_1)\rho(d_1,d_2) }{\rho(1,d_2)d_1^2}.
$$
This summand is a multiplicative arithmetic function in $d_1$ and so the
associated Dirichlet series $F_1(s)$ has an Euler product
$\prod_{p}F_{1,p}(s).$  
When $p^{\nu_2}\| d_2$, we have
$$
F_{1,p}(s)= \sum_{\nu_1\geq 0}\frac{\chi( p^{\nu_1})\rho( p^{\nu_1},p^{\nu_2})
}{\rho(1,p^{\nu_2})p^{\nu_1(2+s)}}.
$$ 
In particular when $p\nmid d_2$ we have 
$$
F_{1,p}(s)=\Big(1-\frac{\chi(p)}{p^{1+s}}\Big)^{-1}
$$
since $\rho(d_1,1)=d_1$.
We may therefore write 
$
F_1(s)=L(1+s,\chi)J_1( 1+s;d_2),
$
where $L(1+s,\chi)$ is the Dirichlet $L$-function associated to
$\chi$ and $J_1(s;d_2)$ is the Dirichlet series associated to an arithmetic
function $j_{d_2}$, with $J_1$ absolutely convergent in the 
half-plane $\Re e (s) \geq 0$. 
We observe that
\begin{equation}
\label{H1p/F1p}
J_{1,p}( 1;d_2)=\Big(1-\frac{\chi(p)}{p}\Big)F_{1,p}(0).
\end{equation}

Let us write $ J_1^*(s;d_2)$ for the  Dirichlet series associated to $|j_{d_2}|$.
For any    $A>0$, Lemma \ref{convol}
yields 
$$
S_1 (V_1,d_2)=L(1,\chi)J_1(1;d_2) +O\Big(\frac{J_1^*(\frac{3}{4};d_2)}{(\log
  V_1)^{A}}\Big).
$$
Now it is clear that 
$$
J_1^*\Big(\frac{3}{4}; d_2\Big)=\prod_{p^{\nu_2}\parallel d_2}J_{1,p}^*\Big(\frac{3}{4}; p^{\nu_2}\Big),
$$
with
$$
\rho(1,p^{\nu_2})J_{1,p}^*\Big(\frac{3}{4};p^{\nu_2}\Big)
\leq (1+ p^{-\frac{3}{4}})
\sum_{\nu_1\geq 0}\frac{ \rho(
p^{\nu_1},p^{\nu_2}) }{p^{\frac{7\nu_1}{4} }}.
$$ 
We apply the inequalities in Lemma \ref{rhopremier} to estimate $\rho(
p^{\nu_1},p^{\nu_2}) .$ 

Suppose first that $p\nmid c_0\D \Delta'$. Then 
$\rho(1,p^{\nu_2})\leq 4p^{\nu_2+[\frac{\nu_2}{3}]}$,  
$$
\sum_{1\leq \nu_1< \lceil \frac{\nu_2}{3}\rceil}\frac{ \rho(
p^{\nu_1},p^{\nu_2}) }{p^{\frac{7\nu_1}{4}}}\leq \Big(3+\frac{1}{p}\Big)
\sum_{1\leq \nu_1< \lceil \frac{\nu_2}{3}\rceil} p^{  \frac{\nu_1}{4}+
  \nu_2+[\frac{\nu_2}{3}]}\leq 
\Big(3+\frac{1}{p}\Big)\Big[ \frac{\nu_2}{3}\Big] 
p^{   \nu_2+\frac{5}{4}[\frac{\nu_2}{3}]},
$$
and 
$$
\sum_{\nu_1\geq \lceil \frac{\nu_2}{3}\rceil}\frac{ \rho(
p^{\nu_1},p^{\nu_2}) }{p^{\frac{7\nu_1}{4}}}\leq \sum_{\nu_1\geq \lceil \frac{\nu_2}{3}\rceil}
p^{2 \nu_2-\lceil \frac{\nu_2}{3}\rceil -\frac{3\nu_1}{4}}=
\frac{p^{2 \nu_2-\frac{7}{4}\lceil \frac{\nu_2}{3}\rceil 
 }}{1-p^{-\frac{3}{4}}}.
$$  
Thus
\begin{equation}
\begin{split}\label{bound}
\frac{\rho(1,p^{\nu_2})(J_{1,p}^*(\frac{3}{4};p^{\nu_2})-1)}{p^{\nu_2}} 
\leq~&   \Big(
\frac{p^{  \nu_2-\frac{7}{4} \lceil \frac{\nu_2}{3}\rceil
 }}{1-p^{-\frac{3}{4}}}+\Big(3+\frac{1}{p}\Big)\Big[
\frac{\nu_2}{3}\Big]   
p^{   \frac{5}{4} [\frac{\nu_2}{3}]}\Big)(1+p^{-\frac{3}{4}})
\\
&+4p^{[\frac{\nu_2}{3}]-\frac{3}{4}}\\
\ll & (1+\nu_2)p^{\frac{5\nu_2}{12}}.
\end{split}
\end{equation}

Suppose now that $p\mid \gcd(d_2, c_0\D\Delta')$. On the one
hand we have 
$$
\sum_{\nu_1\geq 0}\frac{ \rho(
p^{\nu_1},p^{\nu_2}) }{p^{\frac{7\nu_1}{4}}}
\ll \rho(1,p^{\nu_2}) +
\sum_{\nu_1\geq 1} \frac{p^{\nu_1+2\nu_2}}{p^{\frac{7\nu_1}{4}}}\ll p^{2\nu_2-\frac{3}{4}},
$$
which will suffice for small values of $\nu_2$. On the other hand we
have
$$ 
\sum_{\nu_1\geq 0}\frac{ \rho(
p^{\nu_1},p^{\nu_2}) }{p^{\frac{7\nu_1}{4} }}
\ll  
\sum_{\nu_1\leq \frac{2\nu_2}{3}}\frac{  p^{2\nu_1+\frac{4\nu_2}{3}}}{p^{\frac{7\nu_1}{4}}}
+\sum_{\nu_1> \frac{2\nu_2}{3}}\frac{  p^{\nu_1+2\nu_2}}{p^{\frac{7\nu_1}{4}}}
\ll p^{  \frac{3\nu_2}{2}}.
$$ 
Observe that
$$
\prod_{p\mid c_0\D\Delta'}\Big(1+O\Big(\sum_{\nu_2\geq 1}\min\{
p^{-\frac{3}{4}},
p^{-\frac{\nu_2}{2}}\}\Big)\Big)
\leq \prod_{p\mid c_0\D\Delta'}\big(1+O( p^{-\frac{3}{4}}) \big),
$$
which is $O(1)$. 
Using Dirichlet convolution these estimates allow us to conclude that
$$
\sum_{ d_2\leq V_2} \frac{ \rho(1,d_2) J_1^*(\frac{3}{4};d_2)}{ d_2^2} \ll \sum_{ d_2\leq V_2} \frac{ \rho(1,d_2) }{ d_2^2}\ll \log V_2,
$$ 
whence 
$$
S(V_1,V_2)
=\frac{\pi}{4}\sum_{ d_2\leq V_2} \frac{\chi( d_2)\rho(1,d_2)J_1(1;d_2)
}{d_2^2}+O\Big(\frac{ \log V_2}{(\log
  V_1)^{A}}\Big).
$$

The function $J_1(1;d_2)$ is a multiplicative arithmetic function in
$d_2$. Let $p\nmid c_0\D \D'$. We have 
$$|J_{1,p}(1;p^{\nu_2})-1|\leq J_{1,p}^*(1;p^{\nu_2})-1\leq
J_{1,p}^*\Big(\frac{3}{4};p^{\nu_2}\Big)-1 .$$ 
Combining \eqref{H1p/F1p} with  \eqref{bound} allows us to
show that for $1\leq \nu_2\leq 3$  
we have 
$$
\rho(1,p^{\nu_2})J_1(1;p^{\nu_2})=\rho(1,p^{\nu_2})+O(p^{2
  \nu_2-\frac{7}{4}\lceil \frac{\nu_2}{3}\rceil}) 
$$
and for $\nu_2\geq 4$ we have
$$
\rho(1,p^{\nu_2})J_1(1;p^{\nu_2})=\rho(1,p^{\nu_2})+
O\big( (1+\nu_2)p^{\frac{5\nu_2}{12}}\big).
$$
Thus, in terms of Dirichlet convolution, the function
$\chi(d_2)\rho(1,d_2)J_1(1;d_2)d_2^{-1}$ 
is close to $ \chi(d_2)\rho(1,d_2)d_2^{-1}$ and so to $\chi(d_2)\rho_{C(x,1)}(d_2).$
It now follows from Lemmas  \ref{dedekind}, \ref{H}  and \ref{convol} that
$$
S(V_1,V_2)
=\frac{\pi}{4}\vartheta(C(x,1);\chi)K_1'+O\Big(\frac{  \log
V_2}{(\log V_1)^{A}}+\frac{  1}{ (\log V_2)^{A}}\Big),
$$
for any  $A>0$, with
\begin{align*}
K_1'
&=\vartheta(C(x,1);\chi)^{-1}\sum_{ d_2\geq 1} \frac{\chi(
  d_2)\rho(1,d_2)J_1(1;d_2)}{d_2^2} \cr
&=\prod_{p}\Big(\frac{1-\chi(p)p^{-1}}{H_{p,C(x,1)}(1)}
\sum_{\nu_2\geq 0}\frac{\chi( p^{\nu_2})\rho(1,p^{\nu_2})J_1(1;p^{\nu_2})
}{ p^{2\nu_2}}\Big)\cr
&= \prod_{p}\Big(\frac{(1-\chi(p)p^{-1})^2}{H_{p,C(x,1)}(1)}
\sum_{\nu_2\geq 0}\frac{\chi( p^{\nu_2}) 
}{ p^{2\nu_2}}\sum_{\nu_1\geq 0}\frac{\chi(p^{\nu_1}) \rho(
p^{\nu_1},p^{\nu_2}) }{p^{2\nu_1 }}\Big)\\
&=\frac{\pi K'}{4 \vartheta(C(x,1);\chi)}.
\end{align*}
Here we have used \eqref{H1p/F1p} for the penultimate equality. 
This completes the proof
of the lemma in the case $V_1\geq V_2$.

Next we suppose that $V_2\geq V_1$. 
The estimation of $S(V_1,V_2)$ in this case is completely analogous to
the case we have just dealt with apart from a number of minor
technical complications.  
We begin with the expressions
$$
S(V_1,V_2)
=\sum_{ d_1\leq V_1} \frac{\chi( d_1)\rho( d_1,1) }{
  d_1^2}S_2(V_2,d_1), \quad
S_2(V_2,d_1)=
 \sum_{ d_2\leq V_2} \frac{\chi( d_2)\rho(d_1,d_2)
 }{\rho(d_1,1)d_2^2}.
$$
One sees that the sum $S_2 (V_2,d_1)$ again involves a multiplicative
arithmetic function with associated Dirichlet series 
 $F_2(s)=\prod_{p}F_{2,p}(s).$ 
When $p\nmid d_1 $, we have 
$$
F_{2,p}(s)=\sum_{\nu_2\geq 0}\frac{\chi( p^{\nu_2})\rho( 1,p^{\nu_2})
}{ p^{\nu_2(2+s)}}=D_p(1+s)=G_{p,C(x,1)}(1+s)A_p(1+s),
$$
where $D_p(s)$, $G_{p,C(x,1)}(s)$, $ A_p(s)$ are the Eulerian factors
of the Dirichlet series appearing in \eqref{defG}.
When $p^{\nu_1}\| d_1$ and $p\nmid c_0\Delta\Delta'$ it follows from
part (2) of Lemma \ref{rhopremier} and the identity $\rho(p^\nu,1)=p^\nu$ 
that 
$$
|F_{2,p}(s)-1|\leq \sum_{\nu_2\geq 1}\frac{ \rho(
p^{\nu_1},p^{\nu_2}) }{\rho( p^{\nu_1},1)p^{ \nu_2(2+\sigma)}}\ll
p^{-\frac{3}{4}},
$$   
for $\Re e (s)=\sigma\geq -\frac{1}{4}$.
When 
$p^{\nu_1}\| d_1$ and $p\mid c_0\Delta\Delta'$ 
we deduce from part (3) of Lemma~\ref{rhopremier} that
$$
F_{2,p}(s)\ll p^{\frac{3\nu_1 }{8}},
$$
for $\Re e (s)\geq -\frac{1}{4}$.
We may therefore write $F_2(s)=G_{C(x,1)}(1+s,\chi)J_2(1+s;d_1)$
with $G_{C(x,1)}(s,\chi)$ given in \eqref{defHf} and $J_2(s;d_1)$ the
Dirichlet series associated to an arithmetic function
$j_{d_1}$ which is absolutely  convergent in the the half-plane $\Re e
(s) >\frac{5}{6}$.

Lemmas   \ref{dedekind}, \ref{H} and \ref{convol} now yield
$$
S(V_1,V_2)=\vartheta(C(x,1);\chi)\sum_{d_1\leq V_1} \frac{\chi( d_1)\rho( d_1,1)J_2(1;d_1) }{
d_1^2} +O\Big(\frac{1}{(\log V_2)^{A}}\sum_{d_1\leq V_1} \frac{ g(d_1) 
}{d_1}\Big),
$$
with $g$ a multiplicative function satisfying
$$
g(p^\nu)=
\begin{cases}
1+O(p^{-\frac{3}{4}}), & \mbox{if $p\nmid c_0\Delta\Delta'$},\\
O(p^{\frac{3\nu}{8}}) , & \mbox{otherwise}.
\end{cases}
$$
This implies that
$$
S(V_1,V_2)=\vartheta(C(x,1);\chi)\sum_{d_1\leq V_1} \frac{\chi( d_1)J_2(1;d_1) }{ d_1}
+O\Big(\frac{\log V_1}{(\log V_2)^{A}} \Big). 
$$
An application of  Lemma \ref{convol} yields
$$
S(V_1,V_2)=\vartheta(C(x,1);\chi)\frac{\pi}{4}K_2'
+O\Big(\frac{\log
V_1}{(\log V_2)^{A}}+\frac{1}{(\log V_1)^{A}} \Big), 
$$
with 
\begin{align*}
K_2'&=\frac{4}{\pi}\sum_{d_1\in\NN} \frac{\chi( d_1)J_2(1;d_1) }{ d_1}
\cr
&=\prod_{p}\Big(1-\frac{\chi(p)}{p}\Big)
\sum_{\nu_1\geq 0}\frac{\chi( p^{\nu_1})\rho( p^{\nu_1},1)J_2(1;p^{\nu_1})
}{ p^{2\nu_1}}
\cr
&= \prod_{p}\Big(\frac{1-\chi(p)p^{-1}}{G_{p,C(x,1)}(1,\chi)}\Big)
\sum_{\nu_1\geq 0}\frac{\chi( p^{\nu_1}) 
}{ p^{2\nu_1}}\sum_{\nu_2\geq 0}\frac{\chi(p^{\nu_2}) \rho(
p^{\nu_1},p^{\nu_2}) }{p^{2\nu_2 }}\cr
&=\frac{\pi K'}{4
\vartheta(C(x,1);\chi)}.
\end{align*}
This completes the proof of the lemma in the remaining case $V_2\geq
V_1$.

It remains to say a few words about the case in which $L,C$ are not
primitive.  Suppose that $L =\ell_1L^*$ and $C=\ell_2C^*$ for 
primitive forms $L^*$ and $C^*$. Then it follows from \eqref{eq:mars}
that 
$$
S(V_1,V_2)=\sum_{h_i\mid \ell_i}\chi(h_1h_2)
S_{\frac{\ell_1}{h_1},\frac{\ell_2}{h_2}}\Big(\frac{V_1}{h_1},\frac{V_2}{h_2}\Big),
 $$
where the inner sum now involves  $L^*, C^*$ and for 
any $\ma{a}\in \NN^2$ we denote by $S_{\ma{a}}(V_1,V_2)$ the
corresponding sum in which $\gcd(d_i,a_i)=1$  in the
summation over $\ma{d}$.
 In our case $\ell_1$ and $\ell_2$ may be
viewed as absolute constants. Tracing through the argument above  
we are easily led to an estimate for   
$S_{\ma{a}}(V_1,V_2)$ that generalises the case $a_1=a_2=1$ that we
have already handled. 
Once inserted into the above this therefore suffices to handle the
case in which $L$ or $C$ is not primitive. 
\end{proof}

Combining Lemma \ref{SV1V2} with partial summation gives
 $$ 
S_{\pm,\pm}(X;\ma{k},\alpha)=X^2 \vol(\mcal{R}  )
\frac{\pi^2K'}{2^{8+k_1+k_2'}}+O\Big(\frac{X^2}{(\log X)^4}\Big).
$$
Bringing everything together in \eqref{SXsuma} and \eqref{defSX} 
we may now conclude that
$$
S (X)= \pi^2K  \vol(\mcal{R})  X^2   +
O\big( X^2(\log X)^{ -\eta+\varepsilon}\big),
$$
with
$$
K=K'\sum_{(k_0,k_1,k_2)\in \ZZ_{\geq 0}^3}\frac{n(k_1,k_2)}{2^{2k_0+k_1+k_2'+2}}
=K'K_2,
$$
by \eqref{calculK2}.  This completes the proof of Theorem \ref{mainS}.

\section{Passage to the intermediate torsors}\label{s:UT}

We are now ready to commence our proof of Theorem \ref{main}.
Recall the assumption in \eqref{eq:chat} that $a=-1$ and $f$  
has degree $3$ or $4$,  with an irreducible cubic factor without 
repeated roots.
Thus $
x_2^4f(\frac{x_1}{x_2})=L(\x)C(\x)$ with $L$ of
degree $1$ and $C$ of degree $3$.  We suppose that $L,C$ take the shape 
\eqref{eq:LC}, for appropriate $a,b,c_i\in \ZZ$.  
Let  $\delta=\sqrt{5\max\{|a,|b|,|c_i|\}}$.  
Then we will work with the norm 
$$
  \label{eq:norm}
\|\x\|=\max\{|x_0|,|x_1|,|x_2|,\delta^{-1}|x_3|, \delta^{-1}|x_4|\},
$$
in the definition of the exponential height function 
$H_4$ on $\PP^4(\QQ)$.

In what follows it will be convenient to use the notation $Z^m$ for
the set of primitive vectors in $\ZZ^m$.
Our starting point is \cite[Lemma 2]{tb}, which reveals that 
$$ 
N(B)=\frac{1}{4}\#\Big\{(y,z,t;u,v)\in Z^3\times Z^2: 
\begin{array}{l}
\|(v^2t,uvt,u^2t,y,z)\|\leq B,\\
y^2+z^2=t^2 L(u,v)C(u,v)
\end{array}
\Big\}.
$$
We denote by $\mathcal{T}\subset \AA^5=\Spec \QQ[y,z,t,u,v]$ the
subvariety defined by the equation 
\begin{equation}
  \label{eq:Teq}
  y^2+z^2=t^2 L(u,v)C(u,v),
\end{equation}
together with $(y,z,t)\neq \ma{0}$ and $(u,v)\neq \ma{0}$. Then
$\mathcal{T}$ is a $\mathbb{G}_m^2$-torsor over $X$. 
We have
$
\|(v^2t,uvt,u^2t,y,z)\|=\max\{u^2, v^2\}|t|,
$ 
by our choice of norm function,
for any $(y,z,t;u,v)$ under consideration.
Since there is no solution with $t=0$ we have
\begin{equation}
  \label{eq:tree}
  N(B)=\frac{1}{2}\#\Big\{(y,z,t;u,v)\in (Z^3\times Z^2)\cap \mathcal{T}: 
~0<\max\{u^2,v^2\}t\leq B
\Big\}.
\end{equation}
The overall contribution that arises from $(y,z,t;u,v)$ for which
$L(u,v)C(u,v)$ is zero is clearly $O(1)$, which is satisfactory. 

Let 
\begin{equation}\label{18-D}
\mathfrak{D}=\{d\in\NN: ~p\mid d \Rightarrow p\equiv 1 \bmod{4} \}
\end{equation}
and note that $d_0\in \mathfrak{D}$ for any $d_0\mid d$.
For $m,n\in \NN$ we let 
$$
r(n;m)=\#\{a,b \in \ZZ: n=a^2+b^2,~\gcd(m,a,b)=1\}.
$$
Then $r(n;1)=r(n)$ is the usual $r$-function and $r(y^2n;y)=0$ unless
$y\in \DD$.
Using the M\"obius function to detect the coprimality condition  
we obtain
$$
r(y^2n;y)=\sum_{\substack{k\mid y\\ k\in \DD}} \mu(k)r\Big(\frac{y^2n}{k^2}\Big),
$$
for any $y\in \mcal{D}$.  
Given any $\ve_1,\ve_2\in\{\pm 1\}$ and $T\geq 1$ we define the region
$$
R^{\ve_1,\ve_2}(T)=\Big\{(u,v) \in \RR^2: 
\begin{array}{ll}
|u|,|v|\leq \sqrt{T},\\
\ve_1L(u,v)>0, ~\ve_2C(u,v)>0
\end{array}
\Big\}.
$$
Applying the above it now follows that
$$
N(B)=\frac{1}{2}\sum_{k\in \DD} \mu(k)
\sum_{\substack{t\leq \frac{B}{k}\\ t \in \DD}}
\sum_{\substack{\ve_1,\ve_2 \in\{\pm 1\}\\ \ve_1\ve_2=1}}
\sum_{(u,v)\in Z^2\cap R^{\ve_1,\ve_2}(\frac{B}{kt})}r(t^2 L^+C^+),
$$
where we have written $L^+=\ve_1L$ and $C^+=\ve_2C$.

In what follows 
it will be convenient to write
$\omega(a_1,\ldots,a_k)=\omega(\gcd(a_1,\ldots,a_k))$, where
$\omega(n)=\sum_{p\mid n}1$.
We would now like to break the summand into a part involving $t^2$, a
part involving $L^+$ and a part involving $C^+$. For this we call upon
the following result, which is established along precisely the same
lines as \cite[Lemme 10]{L1L2Q}, where the analogous formula
for the divisor function is established.

\begin{lemma}\label{lem:r-123}
Let $n_1,n_2,n_3\in\NN$. Then we have 
$$
r(n_1 n_2 n_3)
=
\sum_{d_id_j\mid n_k} \frac{\chi(d_1d_2d_3)\mu(d_1)\mu(d_2d_3)}{2^{\omega(d_2d_3,n_2,n_3)+4}}
r\Big(\frac{n_1}{d_2d_3}\Big)
r\Big(\frac{n_2}{d_1d_3}\Big)r\Big(\frac{n_3}{d_1d_2}\Big),
$$
where the indices $\{i,j,k\}$ run over permutations of the set $\{1,2,3\}$.
\end{lemma}

Applying Lemma \ref{lem:r-123}, we conclude that
$$
r(t^2 L^+C^+)=\sum_{d_1d_2\mid t^2}
\sum_{\substack{d_1d_3\mid L\\ d_2d_3\mid C}} 
\frac{\chi(d_1d_2d_3)\mu(d_3)\mu(d_1d_2)}{2^{\omega(d_1d_2,L,C)+4}}
r\Big(\frac{t^2}{d_1d_2}\Big)
r\Big(\frac{L^+}{d_1d_3}\Big)r\Big(\frac{C^+}{d_2d_3}\Big),
$$
Write $d=d_1d_2$ and note that $d\mid t$ for any value of $d$ producing
a non-zero summand. In particular we will only be interested in values
of $d\in \mathfrak{D}$, so that $\chi(d)=1$.  Writing $t=ds$, we
deduce that  
\begin{align*}
N(B)
&=\frac{1}{2^5}\sum_{\substack{dk\leq B\\ d,k \in \DD}} \mu(d)\mu(k) 
\sum_{\substack{s \leq \frac{B}{dk}\\ s \in \DD}} r(ds^2)
\sum_{\substack{\dd\in\NN^3\\ d=d_1d_2}} \chi(d_3)\mu(d_3)
\SS_{\dd}\Big(\frac{B}{dsk}\Big),
\end{align*}
where 
$$
\SS_{\dd}(T)=
\sum_{\substack{\ve_1,\ve_2 \in\{\pm 1\}\\ \ve_1\ve_2=1}}
\sum_{\substack{(u,v)\in Z^2\cap R^{\ve_1,\ve_2}(T)\\ d_1d_3\mid L, ~d_2d_3\mid C}}
\frac{
r(\frac{L^+}{d_1d_3})r(\frac{C^+}{d_2d_3})}{
2^{\omega(d,L, C)}},
$$
for any $T\geq 1$.  Now the inner sum vanishes unless $d_3\mid
\gcd(L(u,v),C(u,v)),$ with $(u,v)$ a primitive integer vector.
In particular it follows that $d_3\mid \Delta$, 
the resultant of $L$ and $C$, whence $d_3=O(1)$. 

For given $d\in\NN$ we let
\begin{equation}\label{16-f}
f_d(n)=\sum_{n=ab}\mu(a)r(db^2).
\end{equation}
We may now write 
$$
N(B)
=\frac{1}{2^5}\sum_{\substack{dn\leq B\\ d,n \in \DD}}
\mu(d)f_d(n)
\sum_{\substack{\dd\in\NN^3\\ d=d_1d_2\\ d_3 \mid \D}} \chi(d_3)\mu(d_3)
\SS_{\dd}\Big(\frac{B}{dn}\Big).
$$
Recycling the observation that any common divisor of $L(u,v)$ and
$C(u,v)$ must divide $\Delta$, we obtain 
\begin{align*}
\SS_{\dd}(T)&=
\sum_{\substack{\ve_1,\ve_2 \in\{\pm 1\}\\ \ve_1\ve_2=1}}
\sum_{k\mid \gcd(\D,d)}\frac{1}{2^{\omega(k)}} 
\sum_{\substack{(u,v)\in Z^2\cap R^{\ve_1,\ve_2}(T)\\ d_1d_3\mid L,
    ~d_2d_3\mid C\\ k=\gcd(d,L,C)}}
r\Big(\frac{L^+}{d_1d_3}\Big)r\Big(\frac{C^+}{d_2d_3}\Big)\\
&=\sum_{\substack{\ve_1,\ve_2 \in\{\pm 1\}\\ \ve_1\ve_2=1}}
\sum_{kk'\mid \gcd(\D,d)}\frac{\mu(k')}{2^{\omega(k)}} 
\sum_{\substack{(u,v)\in Z^2\cap R^{\ve_1,\ve_2}(T)\\ [d_1d_3,kk']\mid L\\
[d_2d_3,kk']\mid C}}
r\Big(\frac{L^+}{d_1d_3}\Big)r\Big(\frac{C^+}{d_2d_3}\Big).
\end{align*}
Finally, we wish to remove the coprimality condition on $(u,v)$ using
the M\"obius function.  
Let us define 
\begin{equation}\label{eq:ell}
L_\ell=\ell L^+=\ell \ve_1 L, \quad
C_\ell=\ell^3 C^+=\ell^3 \ve_2 C
\end{equation}
for any $\ell \in \NN$. It follows  
that the inner sum over $(u,v)$ is equal to
$\sum_{\substack{\ell\leq \sqrt{T}}} \mu(\ell)
\UU(\ell^{-2}T),$
where if 
$\kk=(k,k')$ then
\begin{equation}
  \label{eq:UU}
\UU(T)=\UU_{\dd,\kk,\ell}^{\ve_1,\ve_2}(T)=
\sum_{\substack{(x,y)\in \ZZ^2\cap R^{\ve_1,\ve_2}(T)\\
    [d_1d_3,kk']\mid L_\ell\\ 
[d_2d_3,kk']\mid C_\ell}}
r\Big(\frac{L_\ell(x,y)}{d_1d_3}\Big)
r\Big(\frac{C_\ell(x,y)}{d_2d_3}\Big).
\end{equation}
We may summarise our investigation as follows.

\begin{lemma}\label{N1-first}
There exists an absolute constant $c>0$ such that
\begin{align*}
N(B)
=
\frac{1}{2^5}\sum_{\ell=1}^\infty\mu(\ell)
\sum_{d \in \DD} \mu(d)
&\sum_{\substack{n\leq N\\ n \in \DD}} f_d(n) 
\sum_{\substack{\ve_1,\ve_2 \in\{\pm 1\}\\ \ve_1\ve_2=1}}
\\
&\times
\sum_{\substack{\dd\in\NN^3\\ d=d_1d_2\\ d_3 \mid \D}}
\chi(d_3)\mu(d_3)
\UU\Big(\frac{B}{d\ell^2 n}\Big),
\end{align*}
where $N=\frac{cB}{d^{\frac{5}{4}}\ell}$ and 
$\UU(T)=\UU_{\dd,\kk,\ell}^{\ve_1,\ve_2}(T)$ is given  by \eqref{eq:UU}.
\end{lemma}

\begin{proof}
In view of our preceding manipulations, the statement of the lemma
is obviously true with $N=\frac{B}{d\ell^2}$ in the summation over
$n$. To see that we may take $N=\frac{cB}{d^{\frac{5}{4}}\ell}$ for some absolute
constant $c>0$, we observe that $\UU(T)=0$ unless $d_{1}\ll
\ell^3T^{\frac{3}{2}}$ and $d_2\ll \ell T^{\frac{1}{2}}$.
Taking $T=\frac{B}{d\ell^2 n}$, it follows that
$
d=d_1d_2\ll \frac{B^2}{d^2n^2},
$
whence $d^{\frac{3}{2}} n\ll B$. But we also have $d\ell^2n\leq B$,
whence in fact
$
d^{\frac{5}{4}} \ell n\ll B,
$ 
as required.
\end{proof}

The groundwork is now laid for an investigation of $\UU(T)$ for appropriate
values of the parameters. In effect, the thrust of this section has
been concerned with passing from solutions of a single
equation $y^2+z^2=t^2L(u,v)C(u,v)$, to solutions of 
$$
\ell L(u,v)=\delta_1(y_1^2+z_1^2), \quad \ell^3 C(u,v)=\delta_2(y_2^2+z_2^2), 
$$
for varying $\delta_1,\delta_2\in\ZZ$. 
This corresponds to a simple descent process and the
pair  of equations defines an intermediate torsor above the
Ch\^atelet surface $X$.

\section{Analysis of $\UU(T)$}

In this section we will study  
$\UU(T)=\UU_{\dd,\kk,\ell}^{\ve_1,\ve_2}(T)$, 
as given by \eqref{eq:UU}. 
We will work with the sets
\begin{align*}
\mathsf{\Lambda}({\ma{D}})
&=\mathsf{\Lambda}({\ma{D}};L,C)=  \{ \x \in
\ZZ^2 : D_1 \mid L (\x), ~ D_2\mid C(\x)\}\\
\mathsf{\Lambda}^*({\ma{D}})
&=\mathsf{\Lambda}^*({\ma{D}};L,C)=  \{ \x \in
\mathsf{\Lambda}({\ma{D}};L,C): \gcd(D_1D_2,\x)=1\},
\end{align*}
for any $\ma{D}\in \NN^2$.
Let us write 
$$
e_1=d_1d_3, \quad e_2=d_2d_3, \quad 
E_1=[d_1d_3,kk'], \quad E_2=[d_2d_3,kk'].
$$
Clearly $e_i,E_i$ are all odd and $e_i \mid E_i$.
Let $\R=R^{\ve_1,\ve_2}(1)$, so that
$\sqrt{T}\R=R^{\ve_1,\ve_2}(T)$. 
We may therefore write 
$$
\UU(T)=\sum_{\substack{\x\in \sfl(\mathbf{E};L_\ell,C_\ell)\cap
    \sqrt{T}\R}}
r\Big(\frac{L_\ell(\x)}{e_1}\Big)r\Big(\frac{C_\ell(\x)}{e_2}\Big),
$$
where $L_\ell, C_\ell$ are given by \eqref{eq:ell}. 
Ultimately we wish to apply Theorem \ref{mainS} to estimate this sum.
However the latter result involves a sum over points of $\ZZ^2$ rather
than points of $\sfl(\mathbf{E};L_\ell,C_\ell)$. 
We will circumvent this difficulty with a  change of variables. 

The first task is to restrict attention to the case in which each
$E_1$ (resp. $E_2$) is coprime to the coefficients of $L_\ell$ (resp. $C_\ell$).
We let $\ell_1,\ell_2 \in \NN$ and $L^*, C^*  $ be primitive forms
such that
$L_\ell =\ell_1L^*$ and $C_\ell=\ell_2C^*.$
In particular $\ell \mid \ell_1, \ell^3\mid \ell_2$ and
$\ell^{-1}\ell_1, \ell^{-3}\ell_2\ll 1$. 
Then 
$ 
\mathsf{\Lambda}(\ma{E};L_\ell,C_\ell)=\mathsf{\Lambda}(\ma{E'};L^*,C^*),$
with 
$$
E_1'=\frac{E_1}{\gcd(E_1,\ell_1)},\quad E_2'=\frac{E_2}{\gcd(E_2,\ell_2)}.
$$
Define the function $\psi:\NN^2\rightarrow\NN$ multiplicatively via
$$
\psi(p^{\alpha_1}, p^{\alpha_2}) =
p^{\max\{\alpha_1,\lceil{\frac{\alpha_2}{3}}\rceil\}}.  
$$
An analysis of what goes on at prime powers easily leads to the
conclusion that 
$$
\mathsf{\Lambda}({\ma{E'}};L^*,C^*)= \bigsqcup_{h\mid
\psi(\ma{E'})}\mathsf{\Lambda}^*({\ma{E''}};L^*,C^*)
=\bigsqcup_{h\mid
\psi(\ma{E'})}\mathsf{\Lambda}^*({\ma{E''}}),
$$  
where 
\begin{equation*} 
E_1''=\frac{E_1'}{\gcd(E_1',h)}, \quad E_2''=\frac{E_2'}{\gcd(E_2',h^3)}.
\end{equation*}
It follows that
$$
\UU(T)=
\sum_{h\mid
\psi(\ma{E'})}
\sum_{\substack{\x\in 
\mathsf{\Lambda}^*({\ma{E''}})\cap 
    b^{-1}\sqrt{T}\R}}
r\Big(\frac{L^*(\x)}{e_1'}\Big)r\Big(\frac{C^*(\x)}{e_2'}\Big),
$$
where 
$$
e_1'=\frac{e_1}{\gcd(e_1,h)},\quad e_2'=\frac{e_2}{\gcd(e_2,h^3)}.
$$
We let $e'=e_1'e_2'$, $E'=E_1'E_2'$ and  $E''=E_1''E_2''$.

In $\mathsf{\Lambda}^*({\ma{E''}})$ we define an equivalence relation
$\x\sim\y$ if and only if there exists $\lambda\in \ZZ$ such that
$$
\x \equiv \lambda  \y \bmod{E''}.
$$ 
Note that any such $\lambda$ must be coprime to $E''$.
This   relation allows us to  partition $\mathsf{\Lambda}^*({\ma{E''}})$ 
into disjoint equivalence classes.  We denote by $\cU({\ma{E''}})$ the
set of these equivalence classes. We claim that
\begin{equation} 
  \label{eq:number}
  \#\cU({\ma{D}}) \ll (D_1D_2D_3)^\ve
\end{equation}
for any $\ma{D}\in \NN^2$. 
To see this we note that 
$$
\#\cU({\ma{D}})=\frac{\rho^*(\ma{D})}{\phi(D_1D_2)}=
\prod_{p^{\nu_i}\| D_i} \frac{\rho^*(p^{\nu_1},p^{\nu_2})}{\phi(p^{\nu_1+\nu_2})},
$$
where $\rho^*(\ma{D})=\rho^*(\ma{D};L^*,C^*)$ is given
multiplicatively as in 
\eqref{eq:star}.  Applying \eqref{eq:linden} we easily deduce
\eqref{eq:number}.

 When $\y \in \cA$ for  $\cA \in
\cU({\ma{E''}})$, we have
$$
\cA = \{\x \in \ZZ^2 : \x \equiv \lambda \y \bmod{E''} \mbox{ with
  $\lambda\in\ZZ$ and $\gcd(\lambda, E'') = 1$}\}.
$$ 
When $\cA \in \cU({\ma{E''}})$ and  $\y_0\in\cA$, we set 
$$
G(\cA) = \{ \x \in \ZZ^2 : \exists \lambda \in \ZZ \mbox{~such that
$\x \equiv \lambda \y_0 \bmod{E''}$} \}.
$$
This defines a sublattice of $\ZZ^2$ of rank $2$
and determinant $E''$. Moreover the definition is independent of $\y_0$.
We conclude that
\begin{equation}
  \label{eq:nebula}
\UU(T)= 
\sum_{h\mid \psi(\ma{E'})}\sum_{\cA\in \cU(\ma{E''})}
\sum_{e\mid E''}\mu(e)S(T,\cA,e)
\end{equation}
where 
$$ 
S(T,\cA,e)=\sum_{\x\in G_e(\cA) \cap
h^{-1}\sqrt{T}\R }
r\Big(\frac{L^*(\x)}{e_1'}\Big)r\Big(\frac{C^*(\x)}{e_2'}\Big),
$$
with 
\begin{align*}
G_e(\cA)&=G (\cA)\cap \{ \x\in \ZZ^2:  e\mid \x\}
= \{ \x\in \ZZ^2: \mbox{$\exists a \in e\ZZ$ s.t. 
$\x \equiv  a \y_0 \bmod{E''}$}  \}.
\end{align*}
We have therefore arrived at summation conditions running over a
lattice $G_e(\cA)$ of determinant 
\begin{equation}
  \label{eq:det-lattice}
\det G_e(\cA)=eE''\gg \frac{de}{\gcd(d,h\ell)}. 
\end{equation}
We are now led to make a change of variables $\x=\mathbf{M}\v$ for any 
$\x\in G_e(\cA)$, where $\mathbf{M}=(\ma{m}_1,\ma{m}_2)$  
is the matrix formed from a minimal basis for the lattice. 
In particular if $s_1\leq s_2$ are the successive minima of $
G_e(\cA)$ with respect to the norm $|\cdot|$, then $s_i=|\ma{m}_i|$
for $i=1,2$ and $s_1s_2$ has order of magnitude $eE''$.
Moreover,  according to Davenport's work in the geometry of numbers \cite[Lemma
5]{dav}, we will have $v_i\ll s_i^{-1}|\x|$ whenever  $\x\in G_e(\cA)$ is
written as $\x= v_1 \mathbf{m}_1+v_2\mathbf{m}_2$.
On defining the region $\R_{\mathbf{M}}=\{
\v\in\RR^2: \mathbf{M}\v\in h^{-1}\R\}$, we observe that
\begin{equation}\label{eq:recall-vol}
\vol( \R_{\mathbf{M}})=\frac{\vol( \mcal{R})}{
h^2|\det \mathbf{M}|}= \frac{\vol( 
\mcal{R})}{h^2eE''}.
\end{equation}
We may now write
\begin{equation}
  \label{eq:shoe}
S(T,\cA,e)=
\sum_{\v\in \ZZ^2\cap  \sqrt{T}\mcal{R}_{\ma{M}} }
r(M_1(\v))r(M_2(\v)) 
\end{equation}
with 
$$
M_1(\v)=\frac{L^*({\ma{M}}\v)}{e_1'},\quad
  M_2(\v)=\frac{C^*({\ma{M}}\v)}{e_2'}.
$$

Our analysis of $S(T,\cA,e)$ will now involve two aspects: a uniform
upper bound and an asymptotic formula.  In the
first instance, therefore, we require an upper bound for this sum which is 
uniform in $d=d_1d_2$ and $\ell$. 
Our principal tool will be previous work of the 
authors \cite{nair}, which is concerned with the average order of
arithmetic functions ranging over the values taken by binary forms.
As usual we will allow all of our implied constants to depend upon the
coefficients of the forms $L$ and $C$. 
In particular we have $d_3\ll 1$.
We will establish the following result.

\begin{lemma}\label{lem:UU-upper} 
Let $\ve>0$ and let $d$ be square-free. Then we have
$$
\UU(T) \ll (d \ell)^{\ve} \hcf(d,\ell)\Big(
\frac{T}{d}+T^{\frac{1}{2}+\ve}\Big).
$$
\end{lemma}

\begin{proof}
Let $r_2(n)$ be defined multiplicatively via
$$
r_2(p^j)=
\begin{cases}
1+\chi(p), & \mbox{if $j=1$ and 
$p \nmid 6d d_3\D\D'\ell$,}\\  
(1+j)^2, & \mbox{otherwise},
\end{cases}
$$
where $\Delta, \Delta'$ are as in \eqref{defDelta}.
It follows from \eqref{eq:shoe} that
$$
S(T,\cA,e)\leq 2^4
\sum_{\substack{\v\in \ZZ^2\\
v_1\ll V_1, ~v_2\ll V_2}}
r_2\big(M_1(\v)M_2(\v)\big),
$$
where 
$V_i= (hs_i)^{-1}\sqrt{T}$ for $i=1,2$.

It is obvious that $r_2$ belongs to the class of 
non-negative arithmetic functions considered in \cite{nair}. 
An application of \cite[Corollary 1]{nair}
therefore  reveals that  
\begin{align*}
S(T,\cA,e)
\ll
(d\ell)^\ve (V_1V_2E + V_1^{1+\ve})
&\ll (d \ell)^\ve   \Big(\frac{T}{h^2s_1s_2}E +
\frac{T^{\frac{1}{2}+\ve}}{h s_1}\Big),
\end{align*}
for any $\ve>0$,
where
$$
E=\prod_{p\leq V_2}\Big(1+\frac{\varrho_{M_2(x,1)}(p)\chi(p)}{p}\Big).
$$
It follows from Lemma \ref{dedekind} that
$E\leq A^{\omega(d \ell)} \ll (d\ell)^\ve$ for an
  appropriate constant $A\geq 1$. 
Recalling that $s_1s_2\gg ee''$, 
we therefore conclude from \eqref{eq:det-lattice} that
\begin{align*}
S(T,\cA,e)
&\ll (d \ell)^\ve   \Big(\frac{T\gcd(d,h\ell)}{deh^2} +
\frac{T^{\frac{1}{2}+\ve}}{h}\Big).
\end{align*}
Inserting this into 
\eqref{eq:nebula} now yields
\begin{align*}
\UU(T)
&\ll (d \ell)^\ve 
\sum_{h\mid \psi(\ma{e'})}\frac{\gcd(d,h)}{h}
\#\cU(\ma{e''})
\Big(\frac{T\gcd(d,\ell)}{d} +
T^{\frac{1}{2}+\ve}\Big)\\
&\ll (d \ell)^{\ve} \hcf(d,\ell)\Big(
\frac{T}{d}+T^{\frac{1}{2}+\ve}\Big).
\end{align*}
by \eqref{eq:number}. This completes the proof of Lemma~\ref{lem:UU-upper}.
\end{proof}

We now turn to an asymptotic formula for $\UU(T)=\UU_{\dd,\kk,\ell}^{\ve_1,\ve_2}(T)$, as
given by \eqref{eq:nebula} and \eqref{eq:shoe}. Whereas in the previous
lemma we sought 
uniformity in $d=d_1d_2$ and
$\ell$, we will 
now allow all of our implied constants to depend in any way upon
$d,\ell$ and the coefficients of $L$ and $C$. 
It is clear that 
$\mcal{R}_{\ma{M}} $ and $M_1, M_2$ satisfy the necessary conditions
for an application of Theorem \ref{mainS}. 
Put
$$ 
K_p(\mathbf{M})=  
\Big(1-\frac{\chi(p)}{p}\Big)^2\sum_{\nu_1,\nu_2\geq 0} 
 \frac{\chi(p^{\nu_1+\nu_2}) \rho(p^{\nu_1},p^{\nu_2};M_1,M_2 ) }{p^{2\nu_1+2\nu_2 }}
$$
for $p>2$ and 
$$
K_2(\mathbf{M})=
4\lim_{n\to \infty}2^{-2n}   
\#\left\{\x \in (\ZZ/2^n
\ZZ)^2: 
\begin{array}{l}  M_1(\x)\in  \mcal{E}\bmod{2^n}  \\   M_2(\x)\in
 \mcal{E}\bmod{2^n}\end{array}\right\}.
$$
Then once combined with \eqref{eq:nebula} and \eqref{eq:recall-vol}, Theorem \ref{mainS} leads to
the following result.

\begin{lemma}\label{rlc:main} 
  Let $\varepsilon>0$. Then we have
\begin{align*}
\UU(T)=
\pi^2W^{\ve_1,\ve_2}(\bd,\kk,\ell) \vol(R^{\ve_1,\ve_2}(1))T +
O\big(T(\log T)^{-\eta+\varepsilon}\big), 
\end{align*}
where the implied constant depends on $d,\ell,L,C$, and 
$$ 
W^{\ve_1,\ve_2}(\bd,\kk,\ell)=
\sum_{h\mid \psi(\ma{E'})}\sum_{\cA\in \cU(\ma{E''})}
\sum_{e\mid E''}\frac{\mu(e)}{h^2eE''}
\prod_p K_p(\ma{M}).
$$
\end{lemma}

It will be useful to have an expression for $W(\ma{d},\ell)$ as an
Euler product. 
Following the argument in \cite[\S 6]{L1L2Q} almost verbatim one is
led to the conclusion that 
$$
W^{\ve_1,\ve_2}(\bd,\kk,\ell)=\prod_p W_p^{\ve_1,\ve_2}(\bd,\kk,\ell),
$$
where for $p>2$,
\begin{equation}
  \label{eq:Wp}
 W_p^{\ve_1,\ve_2}(\bd,\kk,\ell)
=\Big(1-\frac{\chi(p)}{p}\Big)^2\sum_{\nu_1,\nu_2\geq 0} 
 \frac{ \chi(p^{\nu_1+\nu_2})\rho(p^{N_1},p^{N_2};L_\ell,C_\ell) }{p^{2N_1+2N_2 }},
\end{equation}
with 
$
N_i=\max\{v_p(E_i),\nu_i+v_p(e_i)\}
$ for $i=1,2$, 
and 
\begin{equation}
  \label{eq:W2}
 W_2^{\ve_1,\ve_2}(\bd,\kk,\ell)=4\lim_{n\to \infty}2^{-2n}   
\#\left\{\x \in (\ZZ/2^n
\ZZ)^2: 
\begin{array}{l}  \ell L(\x)\in  \ve_1d_3\mcal{E}\bmod{2^n}  \\
  \ell^3C(\x)\in 
\ve_1d_3 \mcal{E}\bmod{2^n}\end{array}\right\}.
\end{equation}
We have used here the fact that $d_1\equiv d_2\equiv 1 \bmod{4}$ and
$\ve_1\ve_2=1$. 
In our work we will also need a good upper bound for the constant 
$W^{\ve_1,\ve_2}(\bd,\kk,\ell)$ which is uniform in $d$ and $\ell.$
This is recorded in the following result.

\begin{lemma}\label{lem:upper-sig}
We have $W^{\ve_1,\ve_2}(\bd,\kk,\ell)\ll d^{-\frac{1}{6}+\ve}\ell^\ve$ for any $\ve>0$.
\end{lemma}

\begin{proof}
Building on the above Euler product representation of
$W^{\ve_1,\ve_2}(\bd,\kk,\ell)$, 
it is clear that 
$|W_2^{\ve_1,\ve_2}(\bd,\kk,\ell)| \leq 4$. Thus we focus our attention on the
factors corresponding to odd primes. 
When $p>2$ we deduce from part (3) of Lemma \ref{rhopremier} that 
 \begin{align*}
|W_p^{\ve_1,\ve_2}(\bd,\kk,\ell)|
&\ll \sum_{\nu_1,\nu_2\geq 0}
 \frac{\min\{p^{N_1+2N_2},
 p^{2N_1+\frac{5N_2}{3}}\}}{p^{2N_1+2N_2}} \ll \sum_{\nu_1,\nu_2\geq 0}
\frac{1}{p^{\frac{N_1}{2}+\frac{N_2}{6}}}.
\end{align*}
Suppose that $v_p(d_1)=\delta_1$ and $v_p(d_2)=\delta_2$. Since
$d=d_1d_2$ is square-free we may assume that $\delta_1+\delta_2=1$ if
$p\mid d.$ Moreover $N_1 \geq \delta_1+\nu_1$ and $N_2\geq
\delta_2+\nu_2$.
We conclude that
$$
\prod_{\substack{p\mid d}}
|W_p^{\ve_1,\ve_2}(\bd,\kk,\ell)|
\ll d^\ve
\prod_{\substack{p\mid d}}
p^{-\frac{\delta_1+\delta_2}{6}}
\sum_{\nu_1,\nu_2\geq 0}
p^{-\frac{\nu_1}{2}-\frac{\nu_2}{6}}
\ll d^{-\frac{1}{6}+\ve}.
$$
Taking $N_i\geq \nu_i$ it also follows that 
$
\prod_{\substack{p\mid D}}
|W_p^{\ve_1,\ve_2}(\bd,\kk,\ell)|
\ll D^\ve, 
$
for any odd $D \in \NN$.  
Finally the analysis in the proof of Lemma \ref{SV1V2}, which is based on
repeated applications of Lemma \ref{rhopremier},  furnishes the bound
$$
\prod_{\substack{p\nmid 2d\ell \D\D' c_0}}
|W_p(\bd,\kk,\ell)|\ll (d\ell)^\ve.
$$
Putting everything together therefore concludes the proof of the lemma.
\end{proof}

\section{Concluding steps}

We are now ready to draw to a close our proof of Theorem \ref{main},
for which we begin with some technical estimates.
Recall the definition \eqref{18-D} of the set $\DD$ and the definition
\eqref{16-f}  of the function $f_d(n)$. We will need the following
easy result.

\begin{lemma}\label{r^2}
Let $d\in \mathfrak{D}$ be square-free. Then we have
$$
\sum_{\substack{n\leq x\\ n\in \mathfrak{D}}} \frac{f_d(n)}{n}=
\frac{r(d)\phi^\dagger(d)}{\pi} \Big(\log x
+O\big(\log^3 (2+\omega(d))\big)\Big),
$$
where
$\phi^\dagger(d)=\prod_{p \mid d}(1+\frac{1}{p})^{-1}.$
\end{lemma}

\begin{proof}
The proof of Lemma \ref{r^2} involves a straightforward consideration
of the corresponding Dirichlet series
$F_d(s)=\sum_{n\in\mathfrak{D}} f_d(n)n^{-s}.$
Let $r_0(n)=\frac{1}{4}r(n)$.
It is easy to see that 
$$
F_d(s)=4\sum_{m\in \mathfrak{D}}\frac{\mu(m)}{m^{s}}
\sum_{n\in\mathfrak{D}} \frac{r_0(dn^2)}{n^{s}}, 
$$
Let $\delta=\delta_p=v_p(d)$. Then for square-free $d\in \mathfrak{D}$ we have 
$\delta\in \{0,1\}$ and $\delta=1$ if and only if $p\mid d$ and
$p\equiv 1 \bmod{4}$. We now have 
\begin{align*}
F_d(s)
&=
4\prod_{p\equiv 1\bmod{4}} \Big(1-\frac{1}{p^s}\Big)
\prod_{p\equiv 1\bmod{4}}\sum_{\nu\geq 0}
\frac{1+\delta +2\nu}{p^{\nu s}}\\
&=4\prod_{p\equiv 1\bmod{4}} \Big(\frac{1+p^{-s}}{1-p^{-s}}\Big)
\prod_{p\equiv 1\bmod{4}}
\Big(\frac{1+\delta+(1-\delta)p^{-s}}{1+p^{-s}}\Big)\\
&=\frac{4\zeta(s)L(s,\chi)}{(1+2^{-s})\zeta(2s)}H_d(s),
\end{align*}
where 
$$
H_d(s)=
\prod_{p\mid d}
\Big(\frac{2}{1+p^{-s}}\Big)=
r_0(d)\prod_{p\mid d}\Big(1+\frac{1}{p^{s}}\Big)^{-1}.
$$
Noting that $H_1(s)=1$ we clearly have 
$
F_d(s)=F_1(s)H_d(s).
$

The Dirichlet series $F_1(s)$ is  meromorphic in the region $\Re
e(s)>\frac{1}{2}$, with a simple pole 
at $s=1$. Moreover there is an arithmetic function $h_d(n)$,
arising from the Dirichlet series $H_d(s)$, such that
$f_d=f_1*h_d$.  On applying
a Tauberian theorem one easily deduces that
the statement of Lemma \ref{r^2} is true when $d=1$.  To see the general case we note
that
\begin{align*}
\sum_{n\leq x} \frac{f_d(n)}{n} = \sum_{m\leq x} \frac{h_d(m)}{m}
\sum_{n\leq \frac{x}{m}}\frac{f_1(n)}{n}=
  \sum_{m\leq x} \frac{h_d(m)}{m} \Big(\frac{4\log x}{\pi} +O(\log 2m)\Big).
\end{align*}
Here 
\begin{align*}
\sum_{m=1}^\infty \frac{|h_d(m)|\log 2m}{m} &\leq
r_0(d)\phid(d)^{-1}\Big(1+\sum_{p\mid d}\frac{\log p}{p}\Big)\\
&\ll
r(d)\phid(d) \phid(d)^{-2}\log(2+\omega(d))\\
& \ll
r(d)\phid(d)  \log^3(2+\omega(d)),
\end{align*}
since
$$
\sum_{p\mid d}\frac{\log p}{p} \leq \sum_{j\leq \omega(d)}\frac{\log
  p_j}{p_j}\ll \log(2+\omega(d)).
$$
On inserting this into the previous formula, we therefore 
complete the proof of the lemma since $H_d(1)=r_0(d)\phi^\dag(d)$.
\end{proof}

Building on Lemma \ref{r^2}, we may  record the inequalities
\begin{equation}
  \label{eq:ps+r^2}
\sum_{\substack{n \leq x\\ n \in \mathfrak{D}}}
\frac{|f_d(n)|}{n^\theta} 
\leq x^{1-\theta}
\sum_{\substack{n \leq x\\ n \in \mathfrak{D}}} \frac{|f_d(n)|}{n} 
\ll d^\ve x^{1-\theta}\log x,
\end{equation}
for any $\ve>0$ and $0<\theta\leq 1$.  
For the deduction of Theorem \ref{main}, we wish to incorporate
the asymptotic formula in Lemma \ref{rlc:main}
into our expression for $N(B)$ in Lemma~\ref{N1-first}. 
Note that there is no uniformity in any of the parameters $\bd,\kk,
\ell$ that feature in Lemma \ref{rlc:main}.
Let us set 
$$
S(B)=S_{\bd,\kk,\ell}^{\ve_1,\ve_2}(B)=
\sum_{\substack{n\leq N\\ n \in \DD}} f_d(n) 
\UU\Big(\frac{B}{d\ell^2 n}\Big),
$$
with $N=\frac{cB}{d^{\frac{5}{4}}\ell}$ for some absolute constant $c>0$,
so that
$$
N(B)
=
\frac{1}{2^5}\sum_{\ell=1}^\infty\mu(\ell)
\sum_{d \in \DD} \mu(d)
\sum_{\substack{\ve_1,\ve_2 \in\{\pm 1\}\\ \ve_1\ve_2=1}}
\sum_{\substack{\dd\in\NN^3\\ d=d_1d_2\\ d_3 \mid \D}}
\chi(d_3)\mu(d_3)
\sum_{kk'\mid \gcd(\D,d)}\frac{\mu(k')}{2^{\omega(k)}} S(B).
$$
Let 
$$
E^{\ve_1,\ve_2}(\bd,\kk,\ell)=\frac{1}{B\log B}\Big|S(B)- \frac{\pi W^{\ve_1,\ve_2}(\bd,\kk,\ell)
  \vol(R^{\ve_1,\ve_2}(1))r(d)\phid(d)B\log B}{d\ell^2}\Big|.
$$
Then it follows from Lemmas  \ref{rlc:main} and \ref{r^2} that 
for fixed $\bd,\kk,\ell$ we have 
$$
E^{\ve_1,\ve_2}(\bd,\kk,\ell)\rightarrow 0
$$ 
as $B \rightarrow \infty$. 
On the other hand, we conclude from \eqref{eq:ps+r^2} and
Lemmas \ref{lem:UU-upper} and \ref{lem:upper-sig}, that 
$$
E^{\ve_1,\ve_2}(\bd,\kk,\ell)\ll (d \ell)^{\ve}
\gcd(d,\ell)\Big(\frac{1}{d^2\ell^2}+
\frac{1}{d^{\frac{9}{8}}\ell^\frac{3}{2}}+\frac{1}{d^{\frac{7}{6}}\ell^2}\Big)
\ll (d \ell)^{\ve}
\frac{\gcd(d,\ell)}{d^{\frac{9}{8}}\ell^\frac{3}{2}},
$$
uniformly in $d,\ell$ and $B$.
Note that 
$$
\sum_\ell \sum_d\sum_{\ve_1,\ve_2}\sum_{\bd}
\sum_{\kk}
E^{\ve_1,\ve_2}(\bd,\kk,\ell)\ll 1.
$$
Writing $r_0(n)=\frac{1}{4}r(n)$, it therefore follows from the dominated
convergence of this sum 
that as $B\rightarrow \infty$ we have 
$N(B)\sim c_0 B\log B$, with 
\begin{equation}\begin{split}\label{eq:c0}
c_0=
\frac{\pi }{2^3}\sum_{\ell=1}^\infty\frac{\mu(\ell)}{\ell^2}
&
\sum_{d \in \DD} 
 \frac{\mu(d) r_0(d)\phid(d)}{d}
\sum_{\substack{\ve_1,\ve_2 \in\{\pm 1\}\\ \ve_1\ve_2=1}}
\vol(R^{\ve_1,\ve_2}(1))\\
&\times \sum_{\substack{\dd\in\NN^3\\ d=d_1d_2\\ d_3 \mid \D}}
\chi(d_3)\mu(d_3)
\sum_{kk'\mid \gcd(\D,d)}\frac{\mu(k')}{2^{\omega(k)}} 
W^{\ve_1,\ve_2}(\bd,\kk,\ell).
\end{split}\end{equation}

In order to complete the proof of Theorem \ref{main}, it remains to
show that $c_0=c_X$ is the constant predicted by Peyre \cite{p}. Given
the general strategy in our earlier work \cite{isk}, we will be brief.  
In particular, since $X$ is $\QQ$-rational, it is easy to relate the 
value of the constant to the count on the 
torsor $\mathcal{T}$ 
considered in \eqref{eq:tree}. One finds
that
$$
c_X=\omega_\infty \prod_p \omega_p,
$$
where $\omega_\infty$ and $\omega_p$ denote the local densities
associated to $\mathcal{T}$ taken with respect to the Leray measure. 
Using symmetry to restrict to the quadrant in which $y>0$ and $z>0$,
it follows that
$$
\omega_\infty =2\lim_{B\rightarrow \infty} \frac{1}{B\log B}
\int_{\mathcal{D}} \frac{\d u\d v\d t \d z}{2\sqrt{t^2LC(u,v)-z^2}},
$$
where 
we have set $LC(u,v)=L(u,v) C(u,v)$ and 
$\mathcal{D}$ is the set of $(u,v,t,z)\in \RR^4$ such that 
$$
0<\max\{u^2,v^2\}t\leq B, \quad 0<z<t\sqrt{LC(u,v)}, \quad 1\leq t\leq
B, \quad LC(u,v)>0. 
$$
In view of the familiar formula
$$
\int_0^{\sqrt{S}}\frac{\d s}{\sqrt{S-s^2}}=\frac{\pi}{2},
$$
it readily follows that
$$
\omega_\infty=\frac{\pi}{2}
\sum_{\substack{\ve_1,\ve_2 \in\{\pm 1\}\\ \ve_1\ve_2=1}}
\vol(R^{\ve_1,\ve_2}(1)).
$$
Turning the $p$-adic densities, we have 
$$
\omega_p = \lim_{n\rightarrow \infty} p^{-4n}\left\{
(y,z ,t,u,v)\in \mathcal{T}(\ZZ/p^n\ZZ): p\nmid (u,v), p\nmid(y,z,t)
\right\}.
$$
Recall the definition of $\mathcal{E}$ from \S \ref{s:intro} and the
identities \cite[Eqs. (2.3) and (2.5)]{4linear}.  To
calculate $\omega_2$ we observe that $t$ is odd in any solution to be
counted. Since
there are $2^{n-1}$ odd integers in the interval $[1,2^n]$ it follows
that 
\begin{align*}
\omega_2
&
=\lim_{n\to \infty}2^{-3n-1}   
\#\left\{(u,v,y,z) \in (\ZZ/2^n
\ZZ)^4: 
\begin{array}{l}
LC(u,v)\equiv y^2+z^2 \bmod{2^n}, \\
2\nmid (u,v)
\end{array}
\right\}\\
&
=\lim_{n\to \infty}2^{-2n}   
\#\left\{(u,v) \in (\ZZ/2^n
\ZZ)^2: 
LC(u,v)\in  \mcal{E}\bmod{2^n}, ~2\nmid (u,v)\right\}.
\end{align*}
For any binary form $F\in \ZZ[u,v]$ and prime power $p^e$,
 let 
\begin{equation}
  \label{eq:dents}
\widetilde{\rho}_{F}(p^e)=
p^{-2(e+1)}
\#\left\{ (u,v)\in
 (\ZZ/p^{e+1}\ZZ)^2: ~ p^e\mid F(u,v), ~p\nmid (u,v)\right\}.
\end{equation}
Suppose now that $p\equiv 3 \bmod{4}$.  Then we obtain
\begin{align*}
\omega_p
&
=\lim_{n\to \infty}\frac{1-\frac{1}{p}}{p^{3n}}   
\#\left\{(u,v,y,z) \in (\ZZ/p^n
\ZZ)^4: 
\begin{array}{l}
LC(u,v)\equiv y^2+z^2 \bmod{p^n}, \\
p\nmid (u,v)
\end{array}
\right\}\\
&=
\Big(1-\frac{1}{p^2}\Big)\sum_{\nu\geq 0} 
(-1)^\nu\widetilde{\rho}_{LC}(p^\nu)
\end{align*}
Finally, when $p\equiv 1 \bmod{4}$, we break the cardinality according
to the value of $v_p(t)$. It follows that
$$
\omega_p
= 1-\frac{1}{p^2} +\Big(1-\frac{1}{p}\Big)^2
\sum_{\nu\geq 1} \widetilde{\rho}_{LC}(p^\nu),
$$
in this case.

We now return 
to our expression \eqref{eq:c0} for $c_0$. Carrying out
the summation over $\ell$, finding that
$$
\sum_{\ell=1}^\infty\frac{\mu(\ell)}{\ell^2}
W^{\ve_1,\ve_2}(\bd,\kk,\ell)=\widetilde{W}^{\ve_1,\ve_2}(\bd,\kk)=
\prod_p\widetilde{W}_p^{\ve_1,\ve_2}(\bd,\kk),
$$
for suitable factors $\widetilde{W}_p^{\ve_1,\ve_2}(\bd,\kk)$.
In view of \eqref{eq:W2}, one has
\begin{align*}
 \widetilde{W}_2^{\ve_1,\ve_2}(\bd,\kk)
&=4\lim_{n\to \infty}2^{-2n}   
\#\left\{\x \in (\ZZ/2^n
\ZZ)^2: 
\begin{array}{l}  L(\x)\in  \ve_1d_3\mcal{E}\bmod{2^n},  \\   C(\x)\in
\ve_1d_3 \mcal{E}\bmod{2^n},\\
2\nmid \x
\end{array}\right\}.
\end{align*}
It is clear that for any $\x$ counted here we have 
both $LC(\x) \in \mcal{E} \bmod{2^n}$ and 
$LC(-\x) \in \mcal{E} \bmod{2^n}$. Conversely, if $\x
\in (\ZZ/2^n
\ZZ)^2$ satisfies 
$LC(\x) \in \mcal{E} \bmod{2^n}$, then either 
$L(\x) \in \ve_1d_3\mcal{E} \bmod{2^n}$ or 
$L(\x) \in -\ve_1d_3\mcal{E} \bmod{2^n}$. In this way we conclude that
$$
 \widetilde{W}_2^{\ve_1,\ve_2}(\bd,\kk)
=2\omega_2,
$$
in the above notation. 
Next, when $p>2$ we deduce from
\eqref{eq:Wp} that
$$
 \widetilde{W}_p^{\ve_1,\ve_2}(\bd,\kk)
=\Big(1-\frac{\chi(p)}{p}\Big)^2\sum_{\nu_1,\nu_2\geq 0} 
\chi(p^{\nu_1+\nu_2})\widetilde{\rho}
(p^{N_1},p^{N_2}),
$$
with 
$$
\widetilde{\rho}
(p^{N_1},p^{N_2}) =p^{-2(N_1+N_2+1)}\#
\left\{ \x \in (\ZZ/p^{N_1+N_2+1}\ZZ)^2: 
\begin{array}{l}
p^{N_1}\mid L(\x), ~ p^{N_2}\mid C(\x),\\ 
p\nmid \x
\end{array}
\right\}. 
$$
Thus 
$\widetilde{W}_p^{\ve_1,\ve_2}(\bd,\kk)$ is independent of $\ve_1,\ve_2$
and so $ \widetilde{W}_p^{\ve_1,\ve_2}(\bd,\kk)
= \widetilde{W}_p(\bd,\kk)$, say.

An easy calculation reveals that 
$$
\prod_{p}\frac{1-\frac{\chi(p)}{p}}{1+\frac{\chi(p)}{p}}=
\frac{4}{\pi}\cdot \frac{\pi}{2}=2.
$$
Our work so far has therefore shown that 
$
c_0=
\omega_\infty \omega_2 \tau,
$
with 
\begin{align*}
\tau=
\sum_{d \in \DD} 
 \frac{\mu(d) r_0(d)\phid(d)}{d}
\sum_{\substack{\dd\in\NN^3\\ d=d_1d_2\\ d_3 \mid \D}}
\chi(d_3)\mu(d_3)
\hspace{-0.2cm}
\sum_{kk'\mid \gcd(\D,d)}\frac{\mu(k')}{2^{\omega(k)}} 
\prod_{p>2}\Big(
\frac{1+\frac{\chi(p)}{p}}{1-\frac{\chi(p)}{p}}
\Big)
\widetilde{W}_p(\bd,\kk).
\end{align*}
We may write $\tau=\prod_{p>2}\tau_p$. Our final task in this
paper is to show that $\tau_p=\omega_p$ for each odd prime $p$.

Let $\alpha=v_p(\Delta)$. We will deal here only with the harder case
$\alpha\geq1$, the case $\alpha=0$ being an easy modification. 
Suppose that $p\equiv
3\bmod {4}$. In this case it is clear that
\begin{align*}
\tau_p
&=
\Big(
\frac{1-\frac{1}{p}}{1+\frac{1}{p}}
\Big)
\Big(1+\frac{1}{p}\Big)^2
\sum_{\nu_1,\nu_2\geq 0} 
\sum_{0\leq \delta_3 \leq  1}
(-1)^{\nu_1+\nu_2}\widetilde{\rho}
(p^{\nu_1+\delta_3},
p^{\nu_2+\delta_3})\\
&=
\Big(
1-\frac{1}{p^2}\Big)
\sum_{\mu_1,\mu_2\geq 0} 
\sum_{0\leq \delta_3 \leq 1}
\overline{\rho}(p^{2\mu_1+\delta_3}, p^{2\mu_2+\delta_3}),
\end{align*}
where
$$
\overline{\rho}
(p^{n_1},p^{n_2})=
p^{-2(n_1+n_2+1)}
\#\left\{ (u,v)\in
 (\ZZ/p^{n_1+n_2+1}\ZZ)^2: ~ 
\begin{array}{l}
p^{n_1}\| L(u,v), \\
p^{n_2}\| C(u,v), \\
p\nmid (u,v)
\end{array}
\right\}.
$$
Setting 
$\overline{\rho}
(p^{n})$ for the analogous density in which one has $p^n\| 
LC(u,v)$
instead of the pair of conditions present 
in 
$\overline{\rho}(p^{n_1},p^{n_2})$, one finds that 
$$
\tau_p=
\Big(
1-\frac{1}{p^2}\Big)
\sum_{\mu\geq 0} 
\overline{\rho}(p^{2\mu})=\omega_p,
$$
as required.

Suppose now that $p\equiv 1\bmod{4}$.
Then we have
\begin{align*}
\tau_p=
\Big(
1-\frac{1}{p^2}
\Big)
\sum_{0\leq \delta\leq 1} 
 \frac{(-1)^\delta r_0(p^{\delta})\phid(p^{\delta})}{p^{\delta}}
\sum_{\substack{
\delta_1,\delta_2,\delta_3\in \{0,1\}\\
\delta_1+\delta_2=\delta}}
(-1)^{\delta_3} 
f_p(\delta_1,\delta_2,\delta_3)
\end{align*}
with
\begin{align*}f_p(\delta_1,\delta_2,\delta_3)=
\sum_{\nu_1,\nu_2\geq 0} 
\sum_{
\substack{
\kappa,\kappa'\geq 0\\
\kappa+\kappa'\leq \delta}}
\frac{(-1)^{\kappa'}}{2^{\kappa}} 
\widetilde{\rho}
(p^{N_1},
p^{N_2}) 
\end{align*}
and 
$N_i=\max\{\kappa+\kappa',\nu_i+\delta_i+\delta_3\}$ for 
$i=1,2$.
We claim that
\begin{equation}
  \label{eq:jeudi}
f_p(\delta_1,\delta_2,\delta_3)=
\sum_{\nu_1,\nu_2\geq 0} 
\frac{(\nu_1+1) (\nu_2+1)}{2^{\min\{\delta,N_1',N_2'\}}}
\overline{\rho}(p^{N_1'},p^{N_2'}),
\end{equation}
with 
$N_i'=\nu_i+\delta_i+\delta_3$ for $i=1,2$.
We begin by noting that
\begin{align*}
f_p(\delta_1,\delta_2,\delta_3)
=
\sum_{\nu_1,\nu_2\geq 0} 
\sum_{
\substack{
\kappa,\kappa'\geq 0\\
\kappa+\kappa'\leq \min\{\delta,
N_1',N_2'\}}}
\frac{(-1)^{\kappa'}}{2^{\kappa}}
(\nu_1+1) (\nu_2+1)\overline{\rho}(p^{N_1'},p^{N_2'}).
\end{align*}
But it is clear that 
$$
\sum_{
\substack{
\kappa,\kappa'\geq 0\\
\kappa+\kappa'\leq \min\{\delta,
N_1',N_2'\}}}
\frac{(-1)^{\kappa'}}{2^{\kappa}}
=
\frac{1}{2^{\min\{\delta,N_1',N_2'\}}},
$$
from which the claim follows.

Given \eqref{eq:jeudi} we are now led to consider 
the quantity
\begin{align*} 
f_p(\delta)=
\sum_{\substack{
\delta_1,\delta_2,\delta_3\in \{0,1\}\\
\delta_1+\delta_2=\delta}}
(-1)^{\delta_3} 
\sum_{\nu_1,\nu_2\geq 0} 
\frac{(\nu_1+1) (\nu_2+1)}{2^{\min\{\delta,N_1',N_2'\}}}
\overline{\rho}(p^{N_1'},p^{N_2'}),
\end{align*}
for each $\delta\in \{0,1\}$.
Let $N_i''=\nu_i+ \delta_i$ for $i=1,2$.
We may write 
\begin{align*} 
f_p(\delta)  &=
\sum_{\substack{
 \delta_1,\delta_2\geq 0\\
\delta_1+\delta_2=\delta }}
\sum_{\nu_1,\nu_2\geq 0} \big(
 (\nu_1+1) (\nu_2+1) 
 - \nu_1   \nu_2 
\big)\frac{\overline{\rho}(p^{N_1''},p^{N_2''})}{
2^{\min\{\delta,N_1'',N_2''\}}}\\ 
 &=
 \sum_{\substack{
 \delta_1,\delta_2\geq 0\\
\delta_1+\delta_2=\delta }}
\sum_{\nu_1,\nu_2\geq 0}  
 (\nu_1+\nu_2+1)
 \frac{\overline{\rho}(p^{N_1''},p^{N_2''})}{2^{\min\{\delta,N_1'',N_2''\}}}.  
\end{align*}
When $\delta=1$ and 
$$
\min\{1,v_p(L(\x)),v_p(C(\x))\}=\min\{1,N_1'',N_2''\}\geq
1,
$$
with $p^{\nu +\delta}\| LC(\x)$,  there are
two choices of $(\delta_1,\delta_2)$ such that
$\delta_1+\delta_2=\delta$, $p^{\nu_1+\delta_1}\| L(\x)$, $
p^{\nu_2+\delta_2}\| C(\x)$ and $\nu=\nu_1+\nu_2$.   
Thus
\begin{align*}f_p(\delta)
=\sum_{\nu\geq 0}(\nu+1) \overline{\rho}(\nu+\delta)
=\sum_{\nu\geq 0}    \widetilde{\rho}_{LC}(p^{\nu+\delta}),
\end{align*}
in this case. The same is true when $\delta=0$. 
Recalling that $p^{-1}\phid(p)=(p+1)^{-1}$, 
we deduce that
\begin{align*}
\tau_p
&= 
\Big(1-\frac{1}{p^2}\Big) \sum_{0\leq \delta\leq 1} 
 \frac{(-1)^\delta r_0(p^{\delta})\phid(p^{\delta})}{p^{\delta}} f_p(\delta)
\\
&= 
\Big(1-\frac{1}{p^2}
\Big)\Big\{1+\sum_{\nu\geq 1}
\widetilde{\rho}_{LC}(p^{\nu })
\Big(1-\frac{2}{ p+1}\Big)\Big\}\\
&= \omega_p.
\end{align*}
This completes the proof that the value of the leading 
constant in Theorem \ref{main} agrees
with the prediction of Peyre.

\end{document}